\documentclass[12pt]{amsart}
\usepackage{amsmath,amsthm,amsfonts,amssymb,mathrsfs}
\date{\today}

\usepackage{color}

\newtheorem{theorem}{Theorem}[section]
\newtheorem{proposition}[theorem]{Proposition}
\newtheorem{corollary}[theorem]{Corollary}
\newtheorem{lemma}[theorem]{Lemma}

\theoremstyle{definition}
\newtheorem{example}[theorem]{Example}
\newtheorem{remark}[theorem]{Remark}
\newtheorem{definition}[theorem]{Definition}

\usepackage{hyperref}

\begin{document}

\title[On monoids of injective partial selfmaps almost everywhere
the identity]{On monoids of injective partial selfmaps almost
everywhere the identity}

\author[Ivan~Chuchman]{Ivan~Chuchman}
\address{Department of Mechanics and Mathematics, Ivan Franko Lviv
National University, Universytetska 1, Lviv, 79000, Ukraine}
\email{chuchman\underline{\hskip5pt}\,i@mail.ru}

\author[Oleg~Gutik]{Oleg~Gutik}
\address{Department of Mechanics and Mathematics, Ivan Franko Lviv
National University, Universytetska 1, Lviv, 79000, Ukraine}
\email{o\underline{\hskip5pt}\,gutik@franko.lviv.ua,
ovgutik@yahoo.com}

\keywords{Topological semigroup, semitopological semigroup,
topological inverse semigroup, semigroup of bijective partial
transformations, symmetric inverse semigroup, free semilattice,
ideal, congruence, semigroup with the \textsf{F}-property, Baire
space, hereditary Baire space, embedding.
 }

\subjclass[2010]{Primary 22A15, 20M20. Secondary 06F30, 20M18,
22A26, 54E52, 54H12, 54H15}

\begin{abstract}
In this paper we study the semigroup $\mathscr{I}^{\infty}_\lambda$
of injective partial selfmaps almost everywhere the identity of a
set of infinite cardinality $\lambda$. We describe the Green
relations on $\mathscr{I}^{\infty}_\lambda$, all (two-sided) ideals
and all congruences of the semigroup $\mathscr{I}^{\infty}_\lambda$.
We prove that every Hausdorff hereditary Baire topology $\tau$ on
$\mathscr{I}^{\infty}_\omega$ such that
$(\mathscr{I}^{\infty}_\omega,\tau)$ is a semitopological semigroup
is discrete and describe the closure of the discrete semigroup
$\mathscr{I}^{\infty}_\lambda$ in a topological semigroup. Also we
show that for an infinite cardinal $\lambda$ the discrete semigroup
$\mathscr{I}^{\infty}_\lambda$ does not embed into a compact
topological semigroup and construct two non-discrete Hausdorff
topologies turning $\mathscr{I}^{\infty}_\lambda$ into a topological
inverse semigroup.
\end{abstract}

\maketitle


\section{Introduction and preliminaries}

In this paper all spaces are assumed to be Hausdorff. Furthermore we
shall follow the terminology of \cite{CHK, CP, Engelking1989,
GHKLMS, Ruppert1984}. By $\omega$ we shall denote the first infinite
cardinal and by $|A|$ the cardinality of the set $A$. If $Y$ is a
subspace of a topological space $X$ and $A\subseteq Y$, then by
$\operatorname{cl}_Y(A)$ and $\operatorname{Int}_Y(A)$ we shall
denote the topological closure and the interior of $A$ in $Y$,
respectively.

For a semigroup $S$ we denote the semigroup $S$ with the adjoined
unit by $S^1$ (see \cite{CP}).

An algebraic semigroup $S$ is called \emph{inverse} if for any
element $x\in S$ there exists a unique element $x^{-1}\in S$ (called
the \emph{inverse} of $x$) such that $xx^{-1}x=x$ and
$x^{-1}xx^{-1}=x^{-1}$. If $S$ is an inverse semigroup, then the
function $\operatorname{inv}\colon S\to S$ which assigns to every
element $x$ of $S$ its inverse element $x^{-1}$ is called {\it
inversion}.

If $S$ is an inverse semigroup, then by $E(S)$ we shall denote the
\emph{band} (i.e., the subsemigroup of idempotents) of $S$. If the
band $E(S)$ is a non-empty subset of $S$, then the semigroup
operation on $S$ determines a partial order $\leqslant$ on $E(S)$:
$e\leqslant f$ if and only if $ef=fe=e$. This order is called {\em
natural}. A \emph{semilattice} is a commutative semigroup of
idempotents. A semilattice $E$ is called {\em linearly ordered} or a
\emph{chain} if the semilattice operation induces a linear natural
order on $E$. A \emph{maximal chain} of a semilattice $E$ is a chain
which is properly contained in no other chain of $E$. The Axiom of
Choice implies the existence of maximal chains in any partially
ordered set. According to \cite[Definition~II.5.12]{Petrich1984} a
chain $L$ is called an $\omega$-chain if $L$ is isomorphic to
$\{0,-1,-2,-3,\ldots\}$ with the usual order $\leqslant$. Let $E$ be
a semilattice and $e\in E$. We denote ${\downarrow} e=\{ f\in E\mid
f\leqslant e\}$ and ${\uparrow} e=\{ f\in E\mid e\leqslant f\}$.  By
$(\mathscr{P}_{<\omega}(\lambda),\cup)$ we shall denote the free
semilattice with identity over a cardinal $\lambda\geqslant\omega$,
i.e., $\mathscr{P}_{<\omega}(\lambda)$ is the set of all finite
subsets of $\lambda$ with the binary operation $a\cdot b=a\cup b$,
for $a,b\in\mathscr{P}_{<\omega}(\lambda)$.

If $S$ is a semigroup, then we shall denote by $\mathscr{R}$,
$\mathscr{L}$, $\mathscr{J}$, $\mathscr{D}$ and $\mathscr{H}$ the
Green relations on $S$ (see~\cite{CP}):
\begin{align*}
    &\qquad a\mathscr{R}b \mbox{ if and only if } aS^1=bS^1;\\
    &\qquad a\mathscr{L}b \mbox{ if and only if } S^1a=S^1b;\\
    &\qquad a\mathscr{J}b \mbox{ if and only if } S^1aS^1=S^1bS^1;\\
    &\qquad \mathscr{D}=\mathscr{L}\circ\mathscr{R}=\mathscr{R}\circ\mathscr{L};\\
    &\qquad \mathscr{H}=\mathscr{L}\cap\mathscr{R}.
\end{align*}
The relation $\mathscr{J}$ induced a quasi-order
$\leqslant_\mathscr{J}$ on $S$ as follows:
\begin{equation*}
    a\leqslant_\mathscr{J}b \qquad \mbox{ if and only if } \qquad
    S^1aS^1\subseteq S^1bS^1,
\end{equation*}
for $a,b\in S$. This implies that the inclusion order among
two-sided ideals of $S$ induces a partial order among the
$\mathscr{J}$-equivalence classes:
\begin{equation*}
    J_a\preccurlyeq J_b \qquad \mbox{ if and only if } \qquad
    S^1aS^1\subseteq S^1bS^1,
\end{equation*}
for $a,b\in S$, where by $J_a$ we denote the $\mathscr{J}$-class in
$S$ which contains an element $a\in S$ (see
\cite[Section~2.1]{Howie1995}). Then we may thus  regard
$S/\mathscr{J}$ with the relation $\leqslant$ as a partially ordered
set.

A semigroup $S$ is called \emph{simple} if $S$ does not contain
proper two-sided ideals.

A {\it semitopological} (resp.~\emph{topological}) {\it semigroup}
is a topological space together with a separately (resp. jointly)
continuous semigroup operation. An inverse topological semigroup
with the continuous inversion is called a \emph{topological inverse
semigroup}.

In the remainder of the paper $\lambda$ denotes an infinite
cardinal.

Let $\mathscr{I}_\lambda$ denote the set of all partial one-to-one
transformations of an infinite cardinal $\lambda$ together with the
following semigroup operation: $x(\alpha\beta)=(x\alpha)\beta$ if
$x\in\operatorname{dom}(\alpha\beta)=\{
y\in\operatorname{dom}\alpha\mid
y\alpha\in\operatorname{dom}\beta\}$,  for
$\alpha,\beta\in\mathscr{I}_\lambda$. The semigroup
$\mathscr{I}_\lambda$ is called the \emph{symmetric inverse
semigroup} over the cardinal $\lambda$~(see \cite{CP}). The
symmetric inverse semigroup was introduced by
Wagner~\cite{Wagner1952} and it plays a major role in the theory of
semigroups.

A partial map $\alpha\in\mathscr{I}_\lambda$ is called \emph{almost
everywhere the identity} if the set
$\lambda\setminus\operatorname{dom}\alpha$ is finite and
$(x)\alpha\neq x$ only for finitely many $x\in\lambda$. We denote
\begin{equation*}
 \mathscr{I}^{\infty}_\lambda=\{\alpha\in\mathscr{I}_\lambda\mid\alpha
 \mbox{ is almost everywhere the identity}\}.
\end{equation*}
Obviously, $\mathscr{I}^{\infty}_\lambda$ is an inverse subsemigroup
of the semigroup $\mathscr{I}_\omega$. The semigroup
$\mathscr{I}^{\infty}_\lambda$ is called \emph{the semigroup of
injective partial selfmaps almost everywhere the identity} of
$\lambda$. We shall denote every element $\alpha$ of the semigroup
$\mathscr{I}^{\infty}_\lambda$ by
\begin{equation*}
\left(\left.
    \begin{array}{ccc}
      x_1 & \cdots & x_n \\
      y_1 & \cdots & y_n
    \end{array}
\right| A\right)
\end{equation*}
and this means that the following conditions hold:
\begin{itemize}
  \item[$(i)$] $A$ is the maximal subset of $\lambda$ with the finite
   complement such that $\alpha|_A\colon A\rightarrow A$ is an
   identity map;

  \item[$(ii)$] $\{x_1,\ldots,x_n\}$ and $\{y_1,\ldots,y_n\}$ are
  finite (not necessary non-empty) subsets of $\lambda\setminus A$; \; and

  \item[$(iii)$] $\alpha$ maps $x_i$ into $y_i$ for all
   $i=1,\ldots,n$.
\end{itemize}

We denote the identity of the semigroup
$\mathscr{I}^{\infty}_\lambda$ by $\mathbb{I}$.

Many semigroup theorists have considered topological semigroups of
(continuous) transformations of topological spaces.
Be\u{\i}da~\cite{Beida1980}, Orlov~\cite{Orlov1974, Orlov1974a}, and
Subbiah~\cite{Subbiah1987} have considered semigroup and inverse
semigroup  topologies on semigroups of partial homeomorphisms of
some classes of topological spaces.

Gutik and Pavlyk  \cite{GutikPavlyk2005} considered the special
case of the semigroup $\mathscr{I}_\lambda^n$: an infinite
topological semigroup of $\lambda\times\lambda$-matrix units
$B_\lambda$. They showed that an infinite topological semigroup of
$\lambda\times\lambda$-matrix units $B_\lambda$ does not embed
into a compact topological semigroup and that $B_\lambda$ is
algebraically $h$-closed in the class of topological inverse
semigroups. They also described the Bohr compactification of
$B_\lambda$, minimal semigroup and minimal semigroup inverse
topologies on $B_\lambda$.

Gutik, Lawson and Repov\v{s} \cite{GutikLawsonRepovs2009} introduced
the notion of a semigroup with a tight ideal series and investigated
their closures in semitopological semigroups, in particular, in
inverse semigroups with continuous inversion. As a corollary they
showed that the symmetric inverse semigroup of finite
transformations $\mathscr{I}_\lambda^n$ of infinite cardinal
$\lambda$ is algebraically closed in the class of (semi)topological
inverse semigroups with continuous inversion. They also derived
related results about the nonexistence of (partial)
compactifications of semigroups with a tight ideal series.

Gutik and Reiter \cite{GutikReiter2009} showed that the
topological inverse semigroup $\mathscr{I}_\lambda^n$ is
algebraically $h$-closed in the class of topological inverse
semigroups. They also proved that a topological semigroup $S$ with
countably compact square $S\times S$ does not contain the
semigroup $\mathscr{I}_\lambda^n$ for infinite cardinals $\lambda$
and showed that the Bohr compactification of an infinite
topological semigroup $\mathscr{I}_\lambda^n$ is the trivial
semigroup.

In \cite{GutikReiter2010?} Gutik and Reiter showed that that the
symmetric inverse semigroup of finite transformations
$\mathscr{I}_\lambda^n$ of infinite cardinal $\lambda$ is
algebraically closed in the class of semitopological inverse
semigroups with continuous inversion. Also there they described
all congruences on the semigroup $\mathscr{I}_\lambda^n$ and all
compact and countably compact topologies $\tau$ on
$\mathscr{I}_\lambda^n$ such that $(\mathscr{I}_\lambda^n,\tau)$
is a semitopological semigroup.

Gutik, Pavlyk and Reiter \cite{GutikPavlykReiter2009} showed that a
topological semigroup of finite partial bijections
$\mathscr{I}_\lambda^n$ of an infinite cardinal with a compact
subsemigroup of idempotents is absolutely $H$-closed. They proved
that no Hausdorff countably compact topological semigroup and no
Tychonoff topological semigroup with pseudocompact square contain
$\mathscr{I}_\lambda^n$ as a subsemigroup. They proved that every
continuous homomorphism from a topological semigroup
$\mathscr{I}_\lambda^n$ into a Hausdorff countably compact
topological semigroup or Tychonoff topological semigroup with
pseudocompact square is annihilating. They also gave sufficient
conditions for a topological semigroup $\mathscr{I}_\lambda^1$ to be
non-$H$-closed and showed that the topological inverse semigroup
$\mathscr{I}_\lambda^1$ is absolutely $H$-closed if and only if the
band $E(\mathscr{I}_\lambda^1)$ is compact
\cite{GutikPavlykReiter2009}.

In \cite{GutikRepovs2010?} Gutik and Repov\v{s} studied the
semigroup $\mathscr{I}_{\infty}^{\!\nearrow}(\mathbb{N})$ of partial
cofinite monotone bijective transformations of the set of positive
integers $\mathbb{N}$. They showed that the semigroup
$\mathscr{I}_{\infty}^{\!\nearrow}(\mathbb{N})$ has algebraic
properties similar to the bicyclic semigroup: it is bisimple and all
of its non-trivial group homomorphisms are either isomorphisms or
group homomorphisms. They proved that every locally compact topology
$\tau$ on $\mathscr{I}_{\infty}^{\!\nearrow}(\mathbb{N})$ such that
$(\mathscr{I}_{\infty}^{\!\nearrow}(\mathbb{N}),\tau)$ is a
topological inverse semigroup, is discrete and described the closure
of $(\mathscr{I}_{\infty}^{\!\nearrow}(\mathbb{N}),\tau)$ in a
topological semigroup.

In \cite{ChuchmanGutik2010?} Gutik and Chuchman studied the
semigroup $\mathscr{I}_{\infty}^{\,\Rsh\!\!\!\nearrow}(\mathbb{N})$
of partial co-finite almost monotone bijective transformations of
the set of positive integers $\mathbb{N}$. They showed that the
semigroup $\mathscr{I}_{\infty}^{\,\Rsh\!\!\!\nearrow}(\mathbb{N})$
has algebraic properties similar to the bicyclic semigroup: it is
bisimple and all of its non-trivial group homomorphisms are either
isomorphisms or group homomorphisms. Also they proved that every
Baire topology $\tau$ on
$\mathscr{I}_{\infty}^{\,\Rsh\!\!\!\nearrow}(\mathbb{N})$ such that
$(\mathscr{I}_{\infty}^{\,\Rsh\!\!\!\nearrow}(\mathbb{N}),\tau)$ is
a semitopological semigroup is discrete, described the closure of
$(\mathscr{I}_{\infty}^{\,\Rsh\!\!\!\nearrow}(\mathbb{N}),\tau)$ in
a topological semigroup and constructed non-discrete Hausdorff
semigroup topologies on the semigroup
$\mathscr{I}_{\infty}^{\,\Rsh\!\!\!\nearrow}(\mathbb{N})$.

In this paper we study the semigroup $\mathscr{I}^{\infty}_\lambda$
of injective partial selfmaps almost everywhere the identity of a
set of infinite cardinality $\lambda$. We describe the Green
relations on $\mathscr{I}^{\infty}_\lambda$, all (two-sided) ideals
and all congruences of the semigroup $\mathscr{I}^{\infty}_\lambda$.
We prove that every Hausdorff hereditary Baire topology $\tau$ on
$\mathscr{I}^{\infty}_\omega$ such that
$(\mathscr{I}^{\infty}_\omega,\tau)$ is a semitopological semigroup
is discrete and describe the closure of the discrete semigroup
$\mathscr{I}^{\infty}_\lambda$ in a topological semigroup. Also we
show that for an infinite cardinal $\lambda$ the discrete semigroup
$\mathscr{I}^{\infty}_\lambda$ does not embed into a compact
topological semigroup and construct two non-discrete Hausdorff
topologies turning $\mathscr{I}^{\infty}_\lambda$ into a topological
inverse semigroup.


\section{Algebraic properties of the semigroup
$\mathscr{I}^{\infty}_\lambda$}

The definition of the semigroup $\mathscr{I}^{\infty}_\lambda$
implies the following proposition:

\begin{proposition}\label{proposition-2.1}
A partial map $\alpha\in\mathscr{I}_\lambda$ is an element of the
semigroup $\mathscr{I}^{\infty}_\lambda$ if and only if the
following assertions hold:
\begin{itemize}
    \item[$(i)$] $|\lambda\setminus\operatorname{dom}\alpha|=
     |\lambda\setminus\operatorname{ran}\alpha|$; and

    \item[$(ii)$] there exists a subset
     $A\subseteq\operatorname{dom}\alpha\cap\operatorname{ran}\alpha$
     such that $\lambda\setminus A$ is a finite subset of $\lambda$
     and the restriction $\alpha|_{A}\colon A\rightarrow A$ is the
     identity map.
\end{itemize}
\end{proposition}

\begin{proposition}\label{proposition-2.2} { }
\begin{itemize}
    \item[$(i)$] An element $\alpha$ of the semigroup
         $\mathscr{I}^{\infty}_\lambda$
         is an idempotent if and only if $(x)\alpha=x$ for every
         $x\in\operatorname{dom}\alpha$.

    \item[$(ii)$] If $\varepsilon,\iota\in
          E(\mathscr{I}^{\infty}_\lambda)$,
          then $\varepsilon\leqslant\iota$ if and only if
          $\operatorname{dom}\varepsilon\subseteq
          \operatorname{dom}\iota$.

    \item[$(iii)$] The semilattice
          $E(\mathscr{I}^{\infty}_\lambda)$ is isomorphic to
          $(\mathscr{P}_{<\omega}(\lambda),\cup)$ under
          the mapping $(\varepsilon)h=\lambda\setminus
          \operatorname{dom}\varepsilon$.

    \item[$(iv)$] Every maximal chain in
          $E(\mathscr{I}^{\infty}_\lambda)$ is an $\omega$-chain.

    \item[$(v)$] $\alpha\mathscr{R}\beta$ in
         $\mathscr{I}^{\infty}_\lambda$ if and only if
         $\operatorname{dom}\alpha=\operatorname{dom}\beta$.

    \item[$(vi)$] $\alpha\mathscr{L}\beta$ in
         $\mathscr{I}^{\infty}_\lambda$ if and only if
         $\operatorname{ran}\alpha=\operatorname{ran}\beta$.

    \item[$(vii)$] $\alpha\mathscr{H}\beta$ in
         $\mathscr{I}^{\infty}_\lambda$ if and only if
         $\operatorname{dom}\alpha=\operatorname{dom}\beta$ and
         $\operatorname{ran}\alpha=\operatorname{ran}\beta$.

    \item[$(viii)$] $\alpha\mathscr{D}\beta$ in
         $\mathscr{I}^{\infty}_\lambda$ if and only if
         $|\lambda\setminus\operatorname{dom}\alpha|=
         |\lambda\setminus\operatorname{dom}\beta|$.

    \item[$(ix)$] If $n$ is a non-negative integer, then for every
         $\alpha,\beta\in\mathscr{I}^{\infty}_\lambda$ such that
         $|\lambda\setminus\operatorname{dom}\alpha|=
         |\lambda\setminus\operatorname{dom}\beta|=n$ there exist
         $\gamma,\delta\in\mathscr{I}^{\infty}_\lambda$ such that
         $\alpha=\gamma\cdot\beta\cdot\delta$ and
         $|\lambda\setminus\operatorname{dom}\gamma|=
         |\lambda\setminus\operatorname{dom}\delta|=n$.

    \item[$(x)$] For every non-negative integer $n$ the set
         $I_n=\{\alpha\in\mathscr{I}^{\infty}_\lambda\mid
         |\lambda\setminus\operatorname{dom}\alpha|\geqslant n\}$ is an
         ideal in $\mathscr{I}^{\infty}_\lambda$. Moreover, for every
         ideal $I$ in $\mathscr{I}^{\infty}_\lambda$ there exists an
         integer $n\geqslant 0$ such that $I$ is equal to~$I_n$.

    \item[$(xi)$] $\mathscr{D}=\mathscr{J}$ in
         $\mathscr{I}^{\infty}_\lambda$.

    \item[$(xii)$] If $\lambda_1$ and $\lambda_2$ are infinite
         cardinals such that $\lambda_1\leqslant\lambda_2$ then
         $\mathscr{I}^{\infty}_{\lambda_1}$ is a subsemigroup of the
         semigroup $\mathscr{I}^{\infty}_{\lambda_2}$.

    \item[$(xiii)$]
         $(\mathscr{I}^{\infty}_{\lambda}/\mathscr{J},\preccurlyeq)$
         is an $\omega$-chain for any infinite cardinal $\lambda$.
\end{itemize}
\end{proposition}

\begin{proof}
Statements $(i)-(iv)$ are trivial and they follow from the
definition of the semigroup $\mathscr{I}^{\infty}_\lambda$.

$(v)$ Let be $\alpha,\beta\in\mathscr{I}^{\infty}_\lambda$ such that
$\alpha\mathscr{R}\beta$. Since $\alpha\mathscr{I}^{\infty}_\lambda=
\beta\mathscr{I}^{\infty}_\lambda$ and
$\mathscr{I}^{\infty}_\lambda$ is an inverse semigroup, Theorem~1.17
\cite{CP} implies that $\alpha\mathscr{I}^{\infty}_\lambda=
\alpha\alpha^{-1}\mathscr{I}^{\infty}_\lambda$,
$\beta\mathscr{I}^{\infty}_\lambda=
\beta\beta^{-1}\mathscr{I}^{\infty}_\lambda$ and hence
$\alpha\alpha^{-1}=\beta\beta^{-1}$. Therefore we get that
$\operatorname{dom}\alpha=\operatorname{dom}\beta$.

Conversely, let be $\alpha,\beta\in\mathscr{I}^{\infty}_\lambda$
such that $\operatorname{dom}\alpha=\operatorname{dom}\beta$. Then
$\alpha\alpha^{-1}=\beta\beta^{-1}$. Since
$\mathscr{I}^{\infty}_\lambda$ is an inverse semigroup, Theorem~1.17
\cite{CP} implies that $\alpha\mathscr{I}^{\infty}_\lambda=
\alpha\alpha^{-1}\mathscr{I}^{\infty}_\lambda=
\beta\mathscr{I}^{\infty}_\lambda$ and hence
$\alpha\mathscr{I}^{\infty}_\lambda=
\beta\mathscr{I}^{\infty}_\lambda$.

The proof of statement $(vi)$ is similar to $(v)$.

Statement $(vii)$ follows from $(v)$ and $(vi)$.

$(viii)$ Let $\alpha,\beta\in\mathscr{I}^{\infty}_\lambda$ be such
that $\alpha\mathscr{D}\beta$. Then there exists
$\gamma\in\mathscr{I}^{\infty}_\lambda$ such that
$\alpha\mathscr{L}\gamma$ and $\gamma\mathscr{R}\beta$. Therefore by
statements $(v)$ and $(vi)$ we have that
$\operatorname{ran}\alpha=\operatorname{ran}\gamma$ and
$\operatorname{dom}\gamma=\operatorname{dom}\beta$. Then
Proposition~\ref{proposition-2.1} implies that
$|\lambda\setminus\operatorname{ran}\gamma|=
|\lambda\setminus\operatorname{dom}\gamma|$ and
$|\lambda\setminus\operatorname{ran}\beta|=
|\lambda\setminus\operatorname{dom}\beta|$, and hence we get that
$|\lambda\setminus\operatorname{dom}\alpha|=
|\lambda\setminus\operatorname{dom}\beta|$.

Let $\alpha$ and $\beta$ are elements of the semigroup
$\mathscr{I}^{\infty}_\lambda$ such that
$|\lambda\setminus\operatorname{dom}\alpha|=
|\lambda\setminus\operatorname{dom}\beta|$. Then
Proposition~\ref{proposition-2.1} implies that
$|\lambda\setminus\operatorname{ran}\alpha|=
|\lambda\setminus\operatorname{dom}\alpha|$ and
$|\lambda\setminus\operatorname{ran}\beta|=
|\lambda\setminus\operatorname{dom}\beta|$. Let $A_\alpha$ and
$A_\beta$ be maximal subsets of $\lambda$ such that the sets
$\lambda\setminus A_\alpha$ and $\lambda\setminus A_\beta$ are
finite and the restrictions $\alpha|_{A_\alpha}\colon
A_\alpha\rightarrow A_\alpha$ and $\beta|_{A_\beta}\colon
A_\beta\rightarrow A_\beta$ are identity maps. We put
$A=A_\alpha\cap A_\beta$. Since $\lambda\setminus A_\alpha$ and
$\lambda\setminus A_\beta$ are finite subsets of $\lambda$ we
conclude that $\lambda\setminus A$ is a finite subset of $\lambda$
too. Since $|\lambda\setminus\operatorname{dom}\alpha|=
|\lambda\setminus\operatorname{dom}\beta|<\omega$
Proposition~\ref{proposition-2.1} implies that
\begin{equation*}
    |\operatorname{dom}\alpha\setminus A|=
    |\operatorname{ran}\alpha\setminus A|=
    |\operatorname{dom}\beta\setminus A|=
    |\operatorname{ran}\beta\setminus A|=n
\end{equation*}
for some non-negative integer $n$. If $n=0$, then $\alpha=\beta$.
Suppose that $n\geqslant 1$. Let $\{x_1,\ldots,x_n\}=
\operatorname{ran}\alpha\setminus A$ and $\{y_1,\ldots,y_n\}=
\operatorname{dom}\beta\setminus A$. We define
\begin{equation*}
\gamma= \left(\left.
    \begin{array}{ccc}
      y_1 & \cdots & y_n\\
      x_1 & \cdots & x_n
    \end{array}
\right| A\right).
\end{equation*}
Then by statements $(v)$ and $(vi)$ we have that
$\alpha\mathscr{L}\gamma$ and $\gamma\mathscr{R}\beta$ in
$\mathscr{I}^{\infty}_\lambda$. Hence $\alpha\mathscr{D}\beta$ in
$\mathscr{I}^{\infty}_\lambda$.

$(ix)$ Let $\alpha$ and $\beta$ be arbitrary elements of the
semigroup $\mathscr{I}^{\infty}_\lambda$ such that
$|\lambda\setminus\operatorname{dom}\alpha|=
|\lambda\setminus\operatorname{dom}\beta|=n$ for some non-negative
integer $n$. Let $A_\alpha$ and $A_\beta$ be maximal subsets of
$\lambda$ such that the sets $\lambda\setminus A_\alpha$ and
$\lambda\setminus A_\beta$ are finite and the restrictions
$\alpha|_{A_\alpha}\colon A_\alpha\rightarrow A_\alpha$ and
$\beta|_{A_\beta}\colon A_\beta\rightarrow A_\beta$ are identity
maps. We put $A=A_\alpha\cap A_\beta$. Since $\lambda\setminus
A_\alpha$ and $\lambda\setminus A_\beta$ are finite subsets of
$\lambda$ we conclude that $\lambda\setminus A$ is a finite subset
of $\lambda$ too. Since $|\lambda\setminus\operatorname{dom}\alpha|=
|\lambda\setminus\operatorname{dom}\beta|$ the definition of the
semigroup $\mathscr{I}^{\infty}_\lambda$ implies that
$|\operatorname{dom}\alpha\setminus A|=
|\operatorname{dom}\beta\setminus A|<\omega$. If
$\operatorname{dom}\alpha\setminus A=
\operatorname{dom}\beta\setminus A=\varnothing$ then $\alpha=\beta$
and hence $\alpha=\gamma\cdot\beta\cdot\delta$ for
$\gamma=\delta=\mathbb{I}$. Otherwise we put
$\{x_1,\ldots,x_k\}=\operatorname{dom}\alpha\setminus A$,
$\{y_1,\ldots,y_k\}=\operatorname{dom}\beta\setminus A$,
$b_1=(y_1)\beta,\ldots,b_k=(y_k)\beta$ and
$a_1=(x_1)\alpha,\ldots,a_k=(x_k)\alpha$, for some positive integer
$k$. We define
\begin{equation*}
\gamma= \left(\left.
    \begin{array}{ccc}
      x_1 & \cdots & x_k\\
      y_1 & \cdots & y_k
    \end{array}
\right| A\right)
 \qquad \mbox{ and } \qquad
\delta= \left(\left.
    \begin{array}{ccc}
      b_1 & \cdots & b_k\\
      a_1 & \cdots & a_k
    \end{array}
\right| A\right).
\end{equation*}
Then $\gamma,\delta\in\mathscr{I}^{\infty}_\lambda$,
$|\lambda\setminus\operatorname{dom}\gamma|=
|\lambda\setminus\operatorname{dom}\delta|=n$ and
$\alpha=\gamma\cdot\beta\cdot\delta$.

$(x)$ Let $\alpha$ and $\beta$ be arbitrary elements of the
semigroup $\mathscr{I}^{\infty}_\lambda$. Since $\alpha$ and $\beta$
are injective partial selfmaps almost everywhere the identity of the
cardinal $\lambda$ we conclude that
\begin{equation*}
|\lambda\setminus\operatorname{dom}(\alpha\cdot\beta)|\geqslant
\max\{|\lambda\setminus\operatorname{dom}\alpha|,
|\lambda\setminus\operatorname{dom}\beta|\}.
\end{equation*}
This implies the first assertion of statement $(x)$.

Let $I$ be an ideal in $\mathscr{I}^{\infty}_\lambda$. Then the
definition of the semigroup $\mathscr{I}^{\infty}_\lambda$ implies
that there exists $\alpha\in I$ such that
\begin{equation*}
    |\lambda\setminus\operatorname{dom}\alpha|=
    \min\{|\lambda\setminus\operatorname{dom}\gamma|\mid\gamma\in I\}.
\end{equation*}
Then $|\lambda\setminus\operatorname{dom}\alpha|=n$ for some integer
$n\geqslant 0$. Hence $I\subseteq I_n$ and by statement $(ix)$ we
get that $I_n\subseteq I$. This implies the second assertion of the
statement.

Statement $(xi)$ follows from statement $(ix)$.

$(xii)$ Let
 $
\alpha= \left(\left.
    \begin{array}{ccc}
      x_1 & \cdots & x_n \\
      y_1 & \cdots & y_n
    \end{array}
\right| A\right)
 $
be an arbitrary element of the semigroup
$\mathscr{I}^{\infty}_{\lambda_1}$ and
$B=\lambda_2\setminus\lambda_1$. We put
\begin{equation*}
\widetilde{\alpha}= \left(\left.
    \begin{array}{ccc}
      x_1 & \cdots & x_n \\
      y_1 & \cdots & y_n
    \end{array}
\right| A\cup B\right).
\end{equation*}
Obviously that $\widetilde{\alpha}\in
\mathscr{I}^{\infty}_{\lambda_2}$. Simple verifications show that
the map $h\colon\mathscr{I}^{\infty}_{\lambda_1}\rightarrow
\mathscr{I}^{\infty}_{\lambda_2}$ defined by the formula
$(\alpha)h=\widetilde{\alpha}$ is an isomorphic embedding of the
semigroup $\mathscr{I}^{\infty}_{\lambda_1}$ into
$\mathscr{I}^{\infty}_{\lambda_2}$.

Statement $(xiii)$ follows from items $(viii)$ and $(xi)$.
\end{proof}

Later we shall need the following proposition:

\begin{proposition}\label{proposition-2.3}
Let $\lambda$ be an arbitrary infinite cardinal. Then for every
finite subset $\{x_1,\ldots,x_n\}$ of $\lambda$ the semigroups
$\mathscr{I}^{\infty}_\lambda$ and $\mathscr{I}^{\infty}_\eta$ are
isomorphic for $\eta=\lambda\setminus\{x_1,\ldots,x_n\}$.
\end{proposition}

\begin{proof}
Since $\lambda$ is infinite we conclude that there exists a
bijective map $f\colon\lambda\rightarrow\eta$. Then the bijection
$f$ generates a map $h\colon\mathscr{I}^{\infty}_\lambda\rightarrow
\mathscr{I}^{\infty}_\eta$ such that the following condition holds:
\begin{equation*}
    (\alpha_\lambda)h=\alpha_\eta \quad \mbox{ if and only if }
    \quad ((x)f)\alpha_\eta=((x)\alpha_\lambda)f
    \; \mbox{ for every } \; x\in\lambda,
\end{equation*}
where $\alpha_\lambda\in\mathscr{I}^{\infty}_\lambda$ and
$\alpha_\eta\in\mathscr{I}^{\infty}_\eta$.

Now we shall show that so defined map $h$ is injective. Suppose to
the contrary that there exist distinct elements
$\alpha_\lambda,\beta_\lambda\in\mathscr{I}^{\infty}_\lambda$ such
that $(\alpha_\lambda)h=(\beta_\lambda)h$. We denote
$\alpha_\eta=(\alpha_\lambda)h$ and $\beta_\eta=(\beta_\lambda)h$.
Then $\operatorname{dom}\alpha_\eta=\operatorname{dom}\beta_\eta$
and $\operatorname{ran}\alpha_\eta=\operatorname{ran}\beta_\eta$ and
since $f\colon\lambda\rightarrow\eta$ is a bijective map we conclude
that
$\operatorname{dom}\alpha_\lambda=\operatorname{dom}\beta_\lambda$
and
$\operatorname{ran}\alpha_\lambda=\operatorname{ran}\beta_\lambda$.
Therefore there exists $x\in\operatorname{ran}\alpha_\lambda$ such
that $(x)\alpha_\lambda\neq(x)\beta_\lambda$. Since
$(\alpha_\lambda)h=(\beta_\lambda)h$ we have that
$((x)f)\alpha_\eta=((x)f)\beta_\eta$. But
$((x)f)\alpha_\eta=((x)\alpha_\lambda)f$ and
$((x)f)\beta_\eta=((x)\beta_\lambda)f$ and since the map
$f\colon\lambda\rightarrow\eta$ is bijective we conclude that
$(x)\alpha_\lambda=(x)\beta_\lambda$, a contradiction. The obtained
contradiction implies that the map
$h\colon\mathscr{I}^{\infty}_\lambda\rightarrow
\mathscr{I}^{\infty}_\eta$ is injective.

Let
\begin{equation*}
    \alpha_\eta=
\left(\left.
    \begin{array}{ccc}
      x_1 & \cdots & x_n \\
      y_1 & \cdots & y_n
    \end{array}
\right| A\right)
\end{equation*}
be an arbitrary element of the semigroup
$\mathscr{I}^{\infty}_\eta$, where $A\subseteq\eta$ and
$x_1,\ldots,x_n,y_1,\ldots,y_n\in\eta$. Since the map
$f\colon\lambda\rightarrow\eta$ is bijective we conclude that
\begin{equation*}
    \alpha_\lambda=
\left(\left.
    \begin{array}{ccc}
      (x_1)f^{-1} & \cdots & (x_n)f^{-1} \\
      (y_1)f^{-1} & \cdots & (y_n)f^{-1}
    \end{array}
\right| (A)f^{-1}\right)
\end{equation*}
is a partial bijective map from $\lambda$ into $\lambda$ such that
the sets $\lambda\setminus\operatorname{dom}\alpha_\lambda$ and
$\lambda\setminus\operatorname{ran}\alpha_\lambda$ are finite.
Therefore $\alpha_\lambda\in\mathscr{I}^{\infty}_\lambda$ and hence
the map $h\colon\mathscr{I}^{\infty}_\lambda\rightarrow
\mathscr{I}^{\infty}_\eta$ is bijective.

Now we prove that the map
$h\colon\mathscr{I}^{\infty}_\lambda\rightarrow
\mathscr{I}^{\infty}_\eta$ is a homomorphism. We fix arbitrary
elements
$\alpha_\lambda,\beta_\lambda\in\mathscr{I}^{\infty}_\lambda$ and
denote $\alpha_\eta=(\alpha_\lambda)h$ and
$\beta_\eta=(\beta_\lambda)h$. Then for every
$x\in\operatorname{ran}\alpha_\lambda$ we have that
\begin{equation*}
    \big((x)f\big)(\alpha_\eta\cdot\beta_\eta)=
    \Big(\big((x)f\big)\alpha_\eta\Big)\beta_\eta=
    \Big(\big((x)\alpha_\lambda\big)f\Big)\beta_\eta=
    \Big(\big((x)\alpha_\lambda\big)\beta_\lambda\Big)f=
    \big((x)(\alpha_\lambda\cdot\beta_\lambda)\big)f,
\end{equation*}
and hence
$(\alpha_\lambda\cdot\beta_\lambda)h=\alpha_\eta\cdot\beta_\eta=
(\alpha_\lambda)h\cdot(\beta_\lambda)h$.

Therefore $h$ is an isomorphism from the semigroup
$\mathscr{I}^{\infty}_\lambda$ onto $\mathscr{I}^{\infty}_\eta$.
\end{proof}

\begin{proposition}\label{proposition-2.4}
Let $\lambda$ be an arbitrary infinite cardinal. Then for every
idempotent $\varepsilon$ of the semigroup
$\mathscr{I}^{\infty}_\lambda$ the semigroups
$\mathscr{I}^{\infty}_\lambda(\varepsilon)=\varepsilon\cdot
\mathscr{I}^{\infty}_\lambda\cdot\varepsilon$ and
$\mathscr{I}^{\infty}_\lambda$ are isomorphic.
\end{proposition}

\begin{proof}
Since
\begin{align*}
    \mathscr{I}^{\infty}_\lambda(\varepsilon)=&\,
    \varepsilon\cdot\mathscr{I}^{\infty}_\lambda\cdot\varepsilon=
    \varepsilon\cdot\mathscr{I}^{\infty}_\lambda\cap
    \mathscr{I}^{\infty}_\lambda\cdot\varepsilon=\\
    =&\,\{\alpha\in\mathscr{I}^{\infty}_\lambda\mid
    \operatorname{dom}\alpha\subseteq\operatorname{dom}\varepsilon\}\cap
    \{\alpha\in\mathscr{I}^{\infty}_\lambda\mid
    \operatorname{ran}\alpha\subseteq\operatorname{ran}\varepsilon\}=\\
    =&\,\{\alpha\in\mathscr{I}^{\infty}_\lambda\mid
    \operatorname{dom}\alpha\subseteq\operatorname{dom}\varepsilon
    \mbox{ and }
    \operatorname{ran}\alpha\subseteq\operatorname{ran}\varepsilon\},
\end{align*}
Proposition~\ref{proposition-2.3} implies the assertion of the
proposition.
\end{proof}

\begin{proposition}\label{proposition-2.5}
For every $\alpha,\beta\in\mathscr{I}^{\infty}_\lambda$, both sets
 $
\{\chi\in\mathscr{I}^{\infty}_\lambda\mid \alpha\cdot\chi=\beta\}
 $
 and
 $
\{\chi\in\mathscr{I}^{\infty}_\lambda\mid \chi\cdot\alpha=\beta\}
 $
are finite. Consequently, every right translation and every left
translation by an element of the semigroup
$\mathscr{I}^{\infty}_\lambda$ is a finite-to-one map.
\end{proposition}

\begin{proof}
We denote $S=\{\chi\in\mathscr{I}^{\infty}_\lambda\mid
\alpha\cdot\chi=\beta\}$ and
$T=\{\chi\in\mathscr{I}^{\infty}_\lambda \mid
\alpha^{-1}\cdot\alpha\cdot\chi=\alpha^{-1}\cdot\beta\}$. Then
$S\subseteq T$ and the restriction of any partial map $\chi\in T$ to
$\operatorname{dom}(\alpha^{-1}\cdot\alpha)$ coincides with the
partial map $\alpha^{-1}\cdot\beta$. Since every partial map from
the semigroup $\mathscr{I}^{\infty}_\lambda$ is an injective partial
selfmap almost everywhere the identity we have that there exist
maximal subsets $A_{\alpha^{-1}\alpha}$ and $A_{\alpha^{-1}\beta}$
in $\lambda$ such that the sets $\lambda\setminus
A_{\alpha^{-1}\alpha}$ and $\lambda\setminus A_{\alpha^{-1}\beta}$
are finite and the restrictions
$(\alpha^{-1}\cdot\alpha)|_{A_{\alpha^{-1}\alpha}}\colon
A_{\alpha^{-1}\alpha}\rightarrow A_{\alpha^{-1}\alpha}$ and
$(\alpha^{-1}\cdot\beta)|_{A_{\alpha^{-1}\beta}}\colon
A_{\alpha^{-1}\beta}\rightarrow A_{\alpha^{-1}\beta}$ are identity
maps. We put $A=A_{\alpha^{-1}\alpha}\cap A_{\alpha^{-1}\beta}$.
Then the definition of the semigroup $\mathscr{I}^{\infty}_\lambda$
implies that the restrictions $(\alpha^{-1}\cdot\alpha)|_{A}\colon
A\rightarrow A$ and $(\alpha^{-1}\cdot\beta)|_{A}\colon A\rightarrow
A$ are identity maps and the set $\lambda\setminus A$ is finite.
This implies that the set $T$ is finite and hence the set $S$ is
finite too.
\end{proof}

For an arbitrary non-zero cardinal $\lambda$ we denote by
$\emph{\textsf{S}}_\infty(\lambda)$ the group of all bijective
transformations of $\lambda$ with finite supports (i.e.,
$\alpha\in\emph{\textsf{S}}_\infty(\lambda)$ if and only if the set
$\{x\in \lambda\mid(x)\alpha\neq x\}$ is finite).

The definition of the semigroup $\mathscr{I}^{\infty}_\lambda$ and
Proposition~\ref{proposition-2.4} imply the following proposition:

\begin{proposition}\label{proposition-2.6}
Every maximal subgroup of the semigroup
$\mathscr{I}^{\infty}_\lambda$ is isomorphic to
$\textsf{S}_\infty(\lambda)$.
\end{proposition}


\section{On congruences on the semigroup
$\mathscr{I}^{\infty}_\lambda$}

If $\mathfrak{R}$ is an arbitrary congruence on a semigroup $S$,
then we denote by $\Phi_\mathfrak{R}\colon S\rightarrow
S/\mathfrak{R}$ the natural homomorphisms from $S$ onto
$S/\mathfrak{R}$. Also we denote by $\Omega_S$ and $\Delta_S$ the
\emph{universal} and the \emph{identity} congruences, respectively,
on the semigroup $S$, i.~e., $\Omega(S)=S\times S$ and
$\Delta(S)=\{(s,s)\mid s\in S\}$.

The following lemma follows from the definition of a congruence on a
semilattice:

\begin{lemma}\label{lemma-3.1}
Let $\mathfrak{R}$ is an arbitrary congruence on a semilattice $E$.
Let $a$ and $b$ be elements of the semilattice $E$ such that
$a\mathfrak{R}b$. Then
\begin{itemize}
  \item[$(i)$] $a\mathfrak{R}(ab)$; \; and

  \item[$(ii)$] if $a\leqslant b$ then $a\mathfrak{R}c$ for all $c\in E$
   such that $a\leqslant c\leqslant b$.
\end{itemize}
\end{lemma}

\begin{proposition}\label{proposition-3.2}
Let $\mathfrak{R}$ be an arbitrary congruence on the semigroup
$\mathscr{I}^{\infty}_\lambda$. Let $\varepsilon$ and $\varphi$ be
idempotents of $\mathscr{I}^{\infty}_\lambda$ such that
$\varepsilon\mathfrak{R}\varphi$ and $\varepsilon\leqslant\varphi$.
If
$|\operatorname{dom}\varphi\setminus\operatorname{dom}\varepsilon|=1$
then the following conditions hold:
\begin{itemize}
  \item[$(i)$] $\varphi\mathfrak{R}\iota$ for all idempotents
   $\iota\in{\downarrow}\varphi$; \; and
  \item[$(ii)$] $\varphi\mathfrak{R}\chi$ for all idempotents
   $\chi\in\mathscr{I}^{\infty}_\lambda$ such that
   $|\lambda\setminus\operatorname{dom}\varphi|=
   |\lambda\setminus\operatorname{dom}\chi|$.
\end{itemize}
\end{proposition}

\begin{proof}
$(i)$ First we shall show that $\varphi\mathfrak{R}\psi$ for all
idempotents  $\psi\in{\downarrow}\varepsilon$. By
Proposition~\ref{proposition-2.2}~$(iv)$ there exists a maximal (not
necessary unique) $\omega$-chain $L$ in
$E(\mathscr{I}^{\infty}_\lambda)$ which contains $\varepsilon$ and
$\psi$. Let $L_0=\{\varepsilon_1,\ldots,\varepsilon_n\}$ be a
maximal subchain in $L$ such that
$\psi=\varepsilon_n<\ldots<\varepsilon_1=\varepsilon$, where $n$ is
some positive integer. The existence of the subchain $L$ follows
from Proposition~\ref{proposition-2.2}~$(iv)$ too. Let
\begin{equation*}
    x_n=\operatorname{dom}\varepsilon_{n-1}\setminus\operatorname{dom}\varepsilon_{n},\,
    x_{n-1}=\operatorname{dom}\varepsilon_{n-2}\setminus\operatorname{dom}\varepsilon_{n-1},\,
    \ldots, \,
    x_2=\operatorname{dom}\varepsilon_{1}\setminus\operatorname{dom}\varepsilon_{2},\,
    x_1=\operatorname{dom}\varphi\setminus\operatorname{dom}\varepsilon_{1}.
\end{equation*}
We put
\begin{equation*}
    \alpha_1=
\left(\left.
\begin{array}{c}
  x_1 \\
  x_2
\end{array}
\right| \operatorname{dom}\varepsilon_{2}\right), \quad
    \alpha_2=
\left(\left.
\begin{array}{c}
  x_2 \\
  x_3
\end{array}
\right| \operatorname{dom}\varepsilon_{3}\right), \quad
 \ldots \,, \quad
  \alpha_{n-1}= \, \left(\left.
\begin{array}{c}
  x_{n-1} \\
  x_{n}
\end{array}
\right| \operatorname{dom}\varepsilon_{n}\right).
\end{equation*}
Then we have that
\begin{equation*}
\begin{array}{ccc}
    \alpha^{-1}_1\cdot\varphi\cdot\alpha_1=\varepsilon_1 &
    \mbox{and} &
    \alpha^{-1}_1\cdot\varepsilon_1\cdot\alpha_1=\varepsilon_2;\\
    \alpha^{-1}_2\cdot\varepsilon_1\cdot\alpha_2=\varepsilon_2 &
    \mbox{and} &
    \alpha^{-1}_2\cdot\varepsilon_2\cdot\alpha_2=\varepsilon_3;\\
     \cdots  & \cdots & \cdots \\
    \alpha^{-1}_{n-1}\cdot\varepsilon_{n-2}\cdot\alpha_{n-1}=\varepsilon_{n-1} &
    \mbox{and}
    &\,\alpha^{-1}_{n-1}\cdot\varepsilon_{n-1}\cdot\alpha_{n-1}=\varepsilon_{n},
\end{array}
\end{equation*}
and hence
 $
    \varepsilon_1\mathfrak{R}\varepsilon_2,
    \varepsilon_2\mathfrak{R}\varepsilon_3, \ldots ,
    \varepsilon_{n-1}\mathfrak{R}\varepsilon_n
 $.
Since $\varphi\mathfrak{R}\varepsilon$ we have that
$\varphi\mathfrak{R}\varepsilon_n$. This completes the proof of the
statement.

Let $\iota$ be an arbitrary idempotent of the semigroup
$\mathscr{I}^{\infty}_\lambda$ such that
$\iota\in{\downarrow}\varphi$. We put
$\iota_0=\varepsilon\cdot\iota$. Then by previous part of the proof
we have that $\iota_0\mathfrak{R}\varphi$ and hence by
Lemma~\ref{lemma-3.1} we get $\iota\mathfrak{R}\varphi$.

$(ii)$ Let $\chi$ be an arbitrary idempotent of the semigroup
$\mathscr{I}^{\infty}_\lambda$ such that $\varphi\neq\chi$ and
$|\lambda\setminus\operatorname{dom}\varphi|=
|\lambda\setminus\operatorname{dom}\chi|$. Then
$\varepsilon\cdot\chi\leqslant\varphi$ and hence by statement $(i)$
we get that $(\varepsilon\cdot\chi)\mathfrak{R}\varphi$. Since
$|\lambda\setminus\operatorname{dom}\varphi|=
|\lambda\setminus\operatorname{dom}\chi|$ we conclude that
$|\operatorname{dom}\varphi\setminus
\operatorname{dom}(\varepsilon\cdot\chi)|=
|\operatorname{dom}\chi\setminus
\operatorname{dom}(\varepsilon\cdot\chi)|$. Let be
$\{x_1,\ldots,x_k\}=\operatorname{dom}\varphi\setminus
\operatorname{dom}(\varepsilon\cdot\chi)$ and
$\{y_1,\ldots,y_k\}=\operatorname{dom}\chi\setminus
\operatorname{dom}(\varepsilon\cdot\chi)$. We put
\begin{equation*}
     \alpha= \,
\left(\left.
\begin{array}{ccc}
  x_1 & \cdots & x_k \\
  y_1 & \cdots & y_k
\end{array}
\right| \operatorname{dom}(\varepsilon\cdot\chi)\right).
\end{equation*}
Then $\alpha^{-1}\cdot\varphi\cdot\alpha=\chi$ and
$\alpha^{-1}\cdot(\varepsilon\cdot\chi)\cdot\alpha=
\varepsilon\cdot\chi$. Therefore we get that
$(\varepsilon\cdot\chi)\mathfrak{R}\chi$ and hence
$\varphi\mathfrak{R}\chi$. This completes the proof of our
statement.
\end{proof}

\begin{theorem}\label{theorem-3.4}
Let $\mathfrak{R}$ be an arbitrary congruence on the semigroup
$\mathscr{I}^{\infty}_\lambda$ and $\varepsilon$ and $\varphi$ be
distinct $\mathfrak{R}$-equivalent idempotents of
$\mathscr{I}^{\infty}_\lambda$. Then $\alpha\mathfrak{R}\varepsilon$
for every $\alpha\in\mathscr{I}^{\infty}_\lambda$ such that
\begin{equation*}
|\lambda\setminus\operatorname{dom}\alpha|\geqslant
\min\left\{|\lambda\setminus\operatorname{dom}\varphi|,
|\lambda\setminus\operatorname{dom}\varepsilon|\right\}.
\end{equation*}
\end{theorem}

\begin{proof}
In the case when $\alpha$ is an idempotent of the semigroup
$\mathscr{I}^{\infty}_\lambda$ the statement of the theorem follows
from Lemma~\ref{lemma-3.1} and Proposition~\ref{proposition-3.2}.

Suppose that $\alpha$ is an arbitrary non-idempotent element of the
semigroup $\mathscr{I}^{\infty}_\lambda$ such that
$|\lambda\setminus\operatorname{dom}\alpha|\geqslant
\max\left\{|\lambda\setminus\operatorname{dom}\varphi|,
|\lambda\setminus\operatorname{dom}\varepsilon|\right\}$. Since
$\mathscr{I}^{\infty}_\lambda$ is an inverse semigroup we have that
$\alpha\cdot\alpha^{-1}\cdot\alpha=\alpha$ and
Propositions~\ref{proposition-2.1} and \ref{proposition-2.2} imply
that
\begin{equation*}
    |\lambda\setminus\operatorname{dom}\alpha|=
    |\lambda\setminus\operatorname{dom}\alpha^{-1}|=
    |\lambda\setminus\operatorname{dom}(\alpha\cdot\alpha^{-1})|=
    |\lambda\setminus\operatorname{dom}(\alpha^{-1}\cdot\alpha)|\geqslant
    \min\left\{|\lambda\setminus\operatorname{dom}\varphi|,
    |\lambda\setminus\operatorname{dom}\varepsilon|\right\}.
\end{equation*}
Hence $(\alpha\cdot\alpha^{-1})\mathfrak{R}\varepsilon$ and by
Proposition~\ref{proposition-3.2} we have that
$(\alpha\cdot\alpha^{-1})\mathfrak{R}\iota$ for every idempotent
$\iota$ of the semigroup $\mathscr{I}^{\infty}_\lambda$ such that
$\iota\in{\downarrow}\varepsilon$. Definition of the semigroup
$\mathscr{I}^{\infty}_\lambda$ implies that for every
$\alpha\in\mathscr{I}^{\infty}_\lambda$ there exists an idempotent
$\varsigma_\alpha\in\mathscr{I}^{\infty}_\lambda$ such that
$\alpha\cdot\varsigma=\varsigma\cdot\alpha=\varsigma\cdot(\alpha\cdot\alpha^{-1})=\varsigma$
for all idempotents $\varsigma\in\mathscr{I}^{\infty}_\lambda$ such
that $\varsigma\in{\downarrow}\varsigma_\alpha$. Let
$\nu=\varsigma_\alpha\cdot\varepsilon$. Then
$(\alpha\cdot\alpha^{-1})\mathfrak{R}\nu$ and
$\alpha\cdot\nu=\nu\cdot\alpha=\nu\cdot(\alpha\cdot\alpha^{-1})=\nu$.
Therefore we get
\begin{equation*}
    (\alpha)\Phi_{\mathfrak{R}}=
    (\alpha\cdot\alpha^{-1}\cdot\alpha)\Phi_{\mathfrak{R}}=
    (\alpha\cdot\alpha^{-1})\Phi_{\mathfrak{R}}\cdot(\alpha)\Phi_{\mathfrak{R}}=
    (\nu)\Phi_{\mathfrak{R}}\cdot(\alpha)\Phi_{\mathfrak{R}}=
    (\nu\cdot\alpha)\Phi_{\mathfrak{R}}= (\nu)\Phi_{\mathfrak{R}}
\end{equation*}
and $\alpha\mathfrak{R}\nu$. Hence we have that
$\alpha\mathfrak{R}\varepsilon$.
\end{proof}

\begin{proposition}\label{proposition-3.5}
Let $\mathfrak{R}$ be an arbitrary congruence on the semigroup
$\mathscr{I}^{\infty}_\lambda$. Let $\varepsilon$ be an idempotent
of $\mathscr{I}^{\infty}_\lambda$ such that
$|\lambda\setminus\operatorname{dom}\varepsilon|\geqslant 1$ and the
following conditions hold:
\begin{itemize}
  \item[$(i)$] there exists an idempotent
   $\varphi\in\mathscr{I}^{\infty}_\lambda$ such
   that $\varepsilon\mathfrak{R}\varphi$ and
   $|\lambda\setminus\operatorname{dom}\varphi|\geqslant
   |\lambda\setminus\operatorname{dom}\varepsilon|$; \, and
  \item[$(ii)$] does not exist an idempotent
   $\psi\in\mathscr{I}^{\infty}_\lambda$ such
   that $\varepsilon\mathfrak{R}\psi$ and
   $|\lambda\setminus\operatorname{dom}\psi|<
   |\lambda\setminus\operatorname{dom}\varepsilon|$.
\end{itemize}
Then there exists no element $\alpha$ of the semigroup
$\mathscr{I}^{\infty}_\lambda$ such that
$\varepsilon\mathfrak{R}\alpha$ and
   $|\lambda\setminus\operatorname{dom}\alpha|<
   |\lambda\setminus\operatorname{dom}\varepsilon|$.
\end{proposition}

\begin{proof}
Suppose to the contrary that there exists
$\alpha\in\mathscr{I}^{\infty}_\lambda$ such that
$\varepsilon\mathfrak{R}\alpha$ and
$|\lambda\setminus\operatorname{dom}\alpha|<
|\lambda\setminus\operatorname{dom}\varepsilon|$. Since
$\mathscr{I}^{\infty}_\lambda$ is an inverse semigroup
Lemma~III.1.1~\cite{Petrich1984} implies that
$\varepsilon\mathfrak{R}\alpha^{-1}$ and hence
$\varepsilon\mathfrak{R}(\alpha\cdot\alpha^{-1})$. But
$|\lambda\setminus\operatorname{dom}(\alpha\cdot\alpha^{-1})|=
|\lambda\setminus\operatorname{dom}\alpha|<
|\lambda\setminus\operatorname{dom}\varepsilon|$, a contradiction.
An obtained contradiction implies the statement of the proposition.
\end{proof}

\begin{proposition}\label{proposition-3.6}
Let $\mathfrak{R}$ be an arbitrary congruence on the semigroup
$\mathscr{I}^{\infty}_\lambda$. Let $\alpha$ and $\beta$ be
non-$\mathscr{H}$-equivalent elements of
$\mathscr{I}^{\infty}_\lambda$ such that $\alpha\mathfrak{R}\beta$.
Then $\gamma\mathfrak{R}\alpha$ for all
$\gamma\in\mathscr{I}^{\infty}_\lambda$ such that
\begin{equation*}
|\lambda\setminus\operatorname{dom}\gamma|\geqslant
\min\left\{|\lambda\setminus\operatorname{dom}\alpha|,
|\lambda\setminus\operatorname{dom}\beta|\right\}.
\end{equation*}
\end{proposition}

\begin{proof}
Since $\alpha$ and $\beta$ are non-$\mathscr{H}$-equivalent elements
of the inverse semigroup $\mathscr{I}^{\infty}_\lambda$ we conclude
that at least one of the following conditions holds:
\begin{itemize}
  \item[$(i)$] $\alpha\cdot\alpha^{-1}\neq\beta\cdot\beta^{-1}$; \,
  \item[$(ii)$] $\alpha^{-1}\cdot\alpha\neq\beta^{-1}\cdot\beta$.
\end{itemize}
Suppose that the case
$\alpha\cdot\alpha^{-1}\neq\beta\cdot\beta^{-1}$ holds. In the other
case the proof is similar. Since
$\mathscr{I}^{\infty}_\lambda$ is an inverse semigroup
Lemma~III.1.1~\cite{Petrich1984} implies that
$\beta^{-1}\mathfrak{R}\alpha^{-1}$ and hence
$(\beta\cdot\beta^{-1})\mathfrak{R}(\alpha\cdot\alpha^{-1})$. Then we
have that
\begin{equation*}
    |\lambda\setminus\operatorname{dom}\alpha|=
    |\lambda\setminus\operatorname{dom}(\alpha\cdot\alpha^{-1})|
    \qquad \mbox{and} \qquad
    |\lambda\setminus\operatorname{dom}\beta|=
    |\lambda\setminus\operatorname{dom}(\beta\cdot\beta^{-1})|
\end{equation*}
and hence the assumptions of the Theorem~\ref{theorem-3.4} hold.
This completes the proof of the proposition.
\end{proof}

\begin{proposition}\label{proposition-3.7}
Let $\mathfrak{R}$ be an arbitrary congruence on the semigroup
$\mathscr{I}^{\infty}_\lambda$. If $\alpha$ and $\beta$ are distinct
$\mathscr{H}$-equivalent elements of $\mathscr{I}^{\infty}_\lambda$
such that $\alpha\mathfrak{R}\beta$, then $\gamma\mathfrak{R}\alpha$
for all $\gamma\in\mathscr{I}^{\infty}_\lambda$ such that
\begin{equation*}
|\lambda\setminus\operatorname{dom}\gamma|>
|\lambda\setminus\operatorname{dom}\alpha|.
\end{equation*}
\end{proposition}

\begin{proof}
Since $\mathscr{I}^{\infty}_\lambda$ is an inverse semigroup
Theorem~2.20~\cite{CP} and
Proposition~\ref{proposition-2.2}~$(viii)$ imply that without loss
of generality we can assume that $\alpha$ and $\beta$ are elements
of a maximal subgroup $H(\varepsilon)$ of
$\mathscr{I}^{\infty}_\lambda$ with unity $\varepsilon$. Since
$(\alpha\cdot\alpha^{-1})\mathfrak{R}(\beta\cdot\alpha^{-1})$ we can
assume that $\alpha$ is an identity of the subgroup
$H(\varepsilon)$. Let $x\in\operatorname{dom}\alpha$ such that
$(x)\beta\neq x$. We put
$\varepsilon_1\colon\operatorname{dom}\alpha\setminus\{x\}
\rightarrow\operatorname{dom}\alpha\setminus\{x\}$ be an identity
map. Then $\varepsilon_1\cdot\alpha=\varepsilon_1$ and
$\operatorname{ran}(\varepsilon_1\cdot\beta)\neq
\operatorname{ran}(\varepsilon_1)$. Therefore by
Proposition~\ref{proposition-2.2}~$(vii)$ we get that the elements
$\varepsilon_1$ and $\varepsilon_1\cdot\beta$ are not
$\mathscr{H}$-equivalent. Since
$|\lambda\setminus\operatorname{dom}\varepsilon_1|=
|\lambda\setminus\operatorname{dom}(\varepsilon_1\cdot\beta)|$ we
have that the assumptions of Proposition~\ref{proposition-3.6} hold.
This completes the proof of the proposition.
\end{proof}

Theorem~\ref{theorem-3.4} and Propositions~\ref{proposition-3.5},
\ref{proposition-3.6} and \ref{proposition-3.7} imply the
following proposition:

\begin{proposition}\label{proposition-3.8}
Let $\mathfrak{R}$ be an arbitrary congruence on the semigroup
$\mathscr{I}^{\infty}_\lambda$. Let $\alpha$ and $\beta$ be distinct
$\mathscr{H}$-equivalent elements of $\mathscr{I}^{\infty}_\lambda$
such that $\alpha\mathfrak{R}\beta$ and suppose that there does not
exist $\gamma\in\mathscr{I}^{\infty}_\lambda$ such that
$\alpha\mathfrak{R}\gamma$ and
$|\lambda\setminus\operatorname{dom}\gamma|<
|\lambda\setminus\operatorname{dom}\alpha|$. Then elements $\mu,
\nu\in\mathscr{I}^{\infty}_\lambda$ with
$|\lambda\setminus\operatorname{dom}\mu|<
|\lambda\setminus\operatorname{dom}\alpha|$ and
$|\lambda\setminus\operatorname{dom}\nu|<
|\lambda\setminus\operatorname{dom}\alpha|$ are
$\mathfrak{R}$-equivalent if and only if $\mu=\nu$.
\end{proposition}

\begin{definition}\label{definition-3.9}
For every non-negative integer $n$ we denote by $\mathfrak{K}_n(I)$
the congruence on the semigroup $\mathscr{I}^{\infty}_\lambda$
generated by the ideal $I_n$, i.~e., $\mathfrak{K}_n(I)=(I_n\times
I_n)\cup\Delta(\mathscr{I}^{\infty}_\lambda)$. We observe that
$\mathfrak{K}_0(I)=\Omega(\mathscr{I}^{\infty}_\lambda)$.
\end{definition}

\begin{remark}\label{remark-3.10}
The group $\textsf{S}_\infty(\lambda)$ has only one non-trivial
normal subgroup: that is a group $\textsf{A}_\infty(\lambda)$ of all
even permutations of the set $\lambda$ (see \cite[pp.~313--314,
Example]{Guran1981} or \cite{KarrasSolitar1956}). Therefore every
non-trivial homomorphism of $\textsf{S}_\infty(\lambda)$ is either
an isomorphism or its image is a two-elements cyclic group.
\end{remark}

\begin{definition}\label{definition-3.11}
Fix an arbitrary non-negative integer $n$. We shall say that
elements $\alpha$ and $\beta$ of the semigroup
$\mathscr{I}^{\infty}_\lambda$ are
\emph{$n_{\textsf{S}_\infty}$-equivalent} if the following
conditions hold:
\begin{itemize}
  \item[$(i)$] $\alpha\mathscr{H}\beta$; \; and
  \item[$(ii)$] $|\lambda\setminus\operatorname{dom}\alpha|=
   |\lambda\setminus\operatorname{dom}\beta|=n$.
\end{itemize}

We define a relation $\mathfrak{K}_n(\textsf{S}_\infty)$ on the
semigroup $\mathscr{I}^{\infty}_\lambda$ as follows:
\begin{equation*}
    \mathfrak{K}_n(\textsf{S}_\infty)=
    \{(\alpha,\beta)\mid
    (\alpha,\beta)\in n_{\textsf{S}_\infty}\}\cup
    (I_{n+1}\times I_{n+1})\cup\Delta(\mathscr{I}^{\infty}_\lambda).
\end{equation*}
Simple verifications show that so defined relation
$\mathfrak{K}_n(\textsf{S}_\infty)$ on
$\mathscr{I}^{\infty}_\lambda$ is an equivalence relation for every
non-negative integer $n$.
\end{definition}

\begin{proposition}\label{proposition-3.12}
The relation $\mathfrak{K}_n(\textsf{S}_\infty)$ is a congruence on
the semigroup $\mathscr{I}^{\infty}_\lambda$.
\end{proposition}

\begin{proof}
First we consider the case when $n=0$. If $\alpha$ and $\beta$ are
distinct elements of the semigroup $\mathscr{I}^{\infty}_\lambda$
such that $\alpha\mathfrak{K}_0(\textsf{S}_\infty)\beta$, then
either $\alpha,\beta\in H(\mathbb{I})$ or $\alpha,\beta\in I_1$.
Suppose that $\alpha,\beta\in H(\mathbb{I})$. Then for every
$\gamma\in\mathscr{I}^{\infty}_\lambda$ we have that either
$\alpha\cdot\gamma,\beta\cdot\gamma\in H(\mathbb{I})$ or
$\alpha\cdot\gamma,\beta\cdot\gamma\in I_1$, and similarly we get
that either $\gamma\cdot\alpha,\gamma\cdot\beta\in H(\mathbb{I})$ or
$\gamma\cdot\alpha,\gamma\cdot\beta\in I_1$. If $\alpha,\beta\in
I_1$ then for every $\gamma\in\mathscr{I}^{\infty}_\lambda$ we have
that $\alpha\cdot\gamma, \beta\cdot\gamma, \alpha\cdot\gamma,
\beta\cdot\gamma\in I_1$. Therefore
$\mathfrak{K}_0(\textsf{S}_\infty)$ is a congruence on the semigroup
$\mathscr{I}^{\infty}_\lambda$.

Suppose that $n$ is an arbitrary positive integer. Let $\alpha$ and
$\beta$ be distinct elements of the semigroup
$\mathscr{I}^{\infty}_\lambda$ such that
$\alpha\mathfrak{K}_n(\textsf{S}_\infty)\beta$. The definition of
the relation $\mathfrak{K}_n(\textsf{S}_\infty)$ implies that only
one of the following conditions holds:
\begin{itemize}
  \item[$(i)$] $|\lambda\setminus\operatorname{dom}\alpha|=
   |\lambda\setminus\operatorname{dom}\beta|=n$; \; or
  \item[$(ii)$] $|\lambda\setminus\operatorname{dom}\alpha|>n$ and
   $|\lambda\setminus\operatorname{dom}\beta|>n$.
\end{itemize}

First we suppose that $|\lambda\setminus\operatorname{dom}\alpha|=
|\lambda\setminus\operatorname{dom}\beta|=n$. Let $\gamma$ be an
arbitrary element of the semigroup $\mathscr{I}^{\infty}_\lambda$.
We consider two cases:
\begin{itemize}
  \item[$a)$]
   $\operatorname{dom}\alpha\subseteq\operatorname{ran}\gamma$; \;
   and
  \item[$b)$]
   $\operatorname{dom}\alpha\nsubseteq\operatorname{ran}\gamma$.
\end{itemize}
Since the elements $\alpha$ and $\beta$ are $\mathscr{H}$-equivalent
in $\mathscr{I}^{\infty}_\lambda$
Proposition~\ref{proposition-2.2}~$(vii)$ implies that in case $a)$
we have that  $\operatorname{dom}(\gamma\cdot\alpha)=
\operatorname{dom}(\gamma\cdot\beta)$ and
$\operatorname{ran}(\gamma\cdot\alpha)=
\operatorname{ran}(\gamma\cdot\beta)$. Then again by
Proposition~\ref{proposition-2.2}~$(vii)$ the elements
$\gamma\cdot\alpha$ and $\gamma\cdot\beta$ are
$\mathscr{H}$-equivalent in $\mathscr{I}^{\infty}_\lambda$. Since
$\operatorname{dom}\alpha\subseteq\operatorname{ran}\gamma$ we get
that $|\lambda\setminus\operatorname{dom}(\gamma\cdot\alpha)|=
|\lambda\setminus\operatorname{dom}(\gamma\cdot\beta)|=n$. Hence we
obtain that
$(\gamma\cdot\alpha)\mathfrak{K}_n(\textsf{S}_\infty)(\gamma\cdot\beta)$.
In case $b)$ we have that $\gamma\cdot\alpha,\gamma\cdot\beta\in
I_{n+1}$ and hence
$(\gamma\cdot\alpha)\mathfrak{K}_n(\textsf{S}_\infty)(\gamma\cdot\beta)$.

The proof the assertion that
$\alpha\mathfrak{K}_n(\textsf{S}_\infty)\beta$ implies
$(\alpha\cdot\delta)\mathfrak{K}_n(\textsf{S}_\infty)(\beta\cdot\delta)$
for every $\delta\in\mathscr{I}^{\infty}_\lambda$ is similar.

Suppose that $|\lambda\setminus\operatorname{dom}\alpha|>n$ and
$|\lambda\setminus\operatorname{dom}\beta|>n$. Then $\alpha,\beta\in
I_{n+1}$. By Proposition~\ref{proposition-2.2}~$(x)$  we have that
$\gamma\cdot\alpha, \gamma\cdot\beta, \alpha\cdot\delta,
\beta\cdot\delta\in I_{n+1}$ and hence
$(\gamma\cdot\alpha)\mathfrak{K}_n(\textsf{S}_\infty)(\gamma\cdot\beta)$
and
$(\alpha\cdot\delta)\mathfrak{K}_n(\textsf{S}_\infty)(\beta\cdot\delta)$
for all $\gamma,\delta\in\mathscr{I}^{\infty}_\lambda$. This
completes the proof of the proposition.
\end{proof}

\begin{definition}\label{definition-3.13}
Fix an arbitrary non-negative integer $n$. We shall say that
elements $\alpha$ and $\beta$ of the semigroup
$\mathscr{I}^{\infty}_\lambda$ are
\emph{$n_{\textsf{A}_\infty}$-equivalent} if the following
conditions hold:
\begin{itemize}
  \item[$(i)$] $\alpha\mathscr{H}\beta$;
  \item[$(ii)$] $\alpha\cdot\beta^{-1}$ is an even permutation of
   the set $\operatorname{dom}\alpha$; \; and
  \item[$(iii)$] $|\lambda\setminus\operatorname{dom}\alpha|=
   |\lambda\setminus\operatorname{dom}\beta|=n$.
\end{itemize}

We define a relation $\mathfrak{K}_n(\textsf{A}_\infty)$ on the
semigroup $\mathscr{I}^{\infty}_\lambda$ as follows:
\begin{equation*}
    \mathfrak{K}_n(\textsf{A}_\infty)=
    \{(\alpha,\beta)\mid
    (\alpha,\beta)\in n_{\textsf{A}_\infty}\}\cup
    (I_{n+1}\times I_{n+1})\cup\Delta(\mathscr{I}^{\infty}_\lambda).
\end{equation*}
Simple verifications show that so defined relation
$\mathfrak{K}_n(\textsf{A}_\infty)$ on
$\mathscr{I}^{\infty}_\lambda$ is an equivalence relation for every
non-negative integer $n$.
\end{definition}

\begin{proposition}\label{proposition-3.14}
The relation $\mathfrak{K}_n(\textsf{A}_\infty)$ is a congruence on
the semigroup $\mathscr{I}^{\infty}_\lambda$.
\end{proposition}

\begin{proof}
First we consider the case when $n=0$. If $\alpha$ and $\beta$ are
distinct elements of the semigroup $\mathscr{I}^{\infty}_\lambda$
such that $\alpha\mathfrak{K}_0(\textsf{S}_\infty)\beta$, then
either $\alpha,\beta\in H(\mathbb{I})$ or $\alpha,\beta\in I_1$.
Suppose that $\alpha,\beta\in H(\mathbb{I})$. Then for every
$\gamma\in H(\mathbb{I})$ we have that $\alpha\cdot\gamma,
\beta\cdot\gamma, \gamma\cdot\alpha, \gamma\cdot\beta\in
H(\mathbb{I})$. Then
$(\alpha\cdot\gamma)\cdot(\beta\cdot\gamma)^{-1}=
\alpha\cdot\gamma\cdot\gamma^{-1}\cdot\beta^{-1}=
\alpha\cdot\beta^{-1}$ is an even permutation of the set $\lambda$.
Also, since $\alpha\cdot\beta^{-1}$ is an even permutation of the
set $\lambda$ we get that
$(\gamma\cdot\alpha)\cdot(\gamma\cdot\beta)^{-1}=
\gamma\cdot\alpha\cdot\beta^{-1}\cdot\gamma^{-1}$ is an even
permutation of the set $\lambda$ too. For every $\gamma\in I_1$ we
have that $\alpha\cdot\gamma, \beta\cdot\gamma, \gamma\cdot\alpha,
\gamma\cdot\beta\in I_1$. If $\alpha,\beta\in I_1$ then for every
$\gamma\in\mathscr{I}^{\infty}_\lambda$ we have that
$\alpha\cdot\gamma, \beta\cdot\gamma, \alpha\cdot\gamma,
\beta\cdot\gamma\in I_1$. Therefore
$\mathfrak{K}_0(\textsf{A}_\infty)$ is a congruence on the semigroup
$\mathscr{I}^{\infty}_\lambda$.

Suppose that $n$ is an arbitrary positive integer. Let $\alpha$ and
$\beta$ be distinct elements of the semigroup
$\mathscr{I}^{\infty}_\lambda$ such that
$\alpha\mathfrak{K}_n(\textsf{A}_\infty)\beta$. The definition of
the relation $\mathfrak{K}_n(\textsf{A}_\infty)$ implies that only
one of the following conditions holds:
\begin{itemize}
  \item[$(i)$] $|\lambda\setminus\operatorname{dom}\alpha|=
   |\lambda\setminus\operatorname{dom}\beta|=n$; \; or
  \item[$(ii)$] $|\lambda\setminus\operatorname{dom}\alpha|>n$ and
   $|\lambda\setminus\operatorname{dom}\beta|>n$.
\end{itemize}

First we suppose that $|\lambda\setminus\operatorname{dom}\alpha|=
|\lambda\setminus\operatorname{dom}\beta|=n$. Let $\gamma$ be an
arbitrary element of the semigroup $\mathscr{I}^{\infty}_\lambda$.
We consider two cases:
\begin{itemize}
  \item[$a)$]
   $\operatorname{dom}\alpha\subseteq\operatorname{ran}\gamma$; \;
   and
  \item[$b)$]
   $\operatorname{dom}\alpha\nsubseteq\operatorname{ran}\gamma$.
\end{itemize}
Suppose case $a)$ holds. Since the elements $\alpha$ and $\beta$ are
$\mathscr{H}$-equivalent in $\mathscr{I}^{\infty}_\lambda$ we have
that Proposition~\ref{proposition-2.2}~$(vii)$ implies that
$\operatorname{dom}(\gamma\cdot\alpha)=
\operatorname{dom}(\gamma\cdot\beta)$ and
$\operatorname{ran}(\gamma\cdot\alpha)=
\operatorname{ran}(\gamma\cdot\beta)$. Then again by
Proposition~\ref{proposition-2.2}~$(vii)$ the elements
$\gamma\cdot\alpha$ and $\gamma\cdot\beta$ are
$\mathscr{H}$-equivalent in $\mathscr{I}^{\infty}_\lambda$. Since
$\operatorname{dom}\alpha\subseteq\operatorname{ran}\gamma$ we get
that $|\lambda\setminus\operatorname{dom}(\gamma\cdot\alpha)|=
|\lambda\setminus\operatorname{dom}(\gamma\cdot\beta)|=n$. We define
a partial map $\gamma_1\colon\lambda\rightharpoonup\lambda$ as
follows $\gamma_1=\gamma|_{(\operatorname{dom}\alpha)\gamma^{-1}}
\colon(\operatorname{dom}\alpha)\gamma^{-1}\rightarrow
\operatorname{dom}\alpha$. Then we get that
$|\lambda\setminus\operatorname{dom}\gamma_1|=
|\lambda\setminus\operatorname{dom}\alpha|=
|\lambda\setminus\operatorname{dom}\beta|=n$,
$\gamma\cdot\alpha=\gamma_1\cdot\alpha$,
$\gamma\cdot\beta=\gamma_1\cdot\beta$ and hence
$(\gamma\cdot\alpha)\cdot(\gamma\cdot\beta)^{-1}=
(\gamma_1\cdot\alpha)\cdot(\gamma_1\cdot\beta)^{-1}=
\gamma_1\cdot\alpha\cdot\beta^{-1}\cdot\gamma_1^{-1}$. Since
$\alpha\cdot\beta^{-1}$ is an even permutation of the set
$\operatorname{dom}\alpha$ we conclude that
$\gamma_1\cdot\alpha\cdot\beta^{-1}\cdot\gamma_1^{-1}$ is an even
permutation of the set $\operatorname{dom}\gamma_1=
(\operatorname{dom}\alpha)\gamma^{-1}$. Hence we obtain that
$(\gamma\cdot\alpha)\mathfrak{K}_n(\textsf{A}_\infty)(\gamma\cdot\beta)$.
In case $b)$ we have that $\gamma\cdot\alpha,\gamma\cdot\beta\in
I_{n+1}$ and hence
$(\gamma\cdot\alpha)\mathfrak{K}_n(\textsf{A}_\infty)(\gamma\cdot\beta)$.

The proof the assertion that
$\alpha\mathfrak{K}_n(\textsf{A}_\infty)\beta$ implies
$(\alpha\cdot\delta)\mathfrak{K}_n(\textsf{A}_\infty)(\beta\cdot\delta)$
for every $\delta\in\mathscr{I}^{\infty}_\lambda$ is similar.

Suppose that $|\lambda\setminus\operatorname{dom}\alpha|>n$ and
$|\lambda\setminus\operatorname{dom}\beta|>n$. Then $\alpha,\beta\in
I_{n+1}$. By Proposition~\ref{proposition-2.2}~$(x)$  we have that
$\gamma\cdot\alpha, \gamma\cdot\beta, \alpha\cdot\delta,
\beta\cdot\delta\in I_{n+1}$ and hence
$(\gamma\cdot\alpha)\mathfrak{K}_n(\textsf{A}_\infty)(\gamma\cdot\beta)$
and
$(\alpha\cdot\delta)\mathfrak{K}_n(\textsf{A}_\infty)(\beta\cdot\delta)$,
for all $\gamma,\delta\in\mathscr{I}^{\infty}_\lambda$. This
completes the proof of the proposition.
\end{proof}

\begin{theorem}\label{theorem-3.15}
The family
\begin{equation*}
\begin{split}
  \textsf{Cong}(\mathscr{I}^{\infty}_\lambda)& \,=
  \{\Delta(\mathscr{I}^{\infty}_\lambda),
    \Omega(\mathscr{I}^{\infty}_\lambda)\}\cup
    \{\mathfrak{K}_n(\textsf{S}_\infty)\mid n=0,1,2,\ldots\}
    \cup\{\mathfrak{K}_n(\textsf{A}_\infty)\mid n=0,1,2,\ldots\}\cup \\
    & \,\cup\{\mathfrak{K}_n(I_n)\mid n=1,2,\ldots\}
\end{split}
\end{equation*}
determines all congruences on the semigroup
$\mathscr{I}^{\infty}_\lambda$.
\end{theorem}

\begin{proof}
Let $\mathfrak{R}$ be non-identity congruence on the semigroup
$\mathscr{I}^{\infty}_\lambda$. Since the set of all non-negative
integers with respect to the usual order $\leqslant$ is well ordered
there exists a minimal non-negative integer $n$ such that there are
two distinct elements $\alpha$ and $\beta$ in
$\mathscr{I}^{\infty}_\lambda$ such that $\alpha\mathfrak{R}\beta$
and $$\min\left\{|\lambda\setminus\operatorname{dom}\alpha|,
|\lambda\setminus\operatorname{dom}\beta|\right\}=n,$$ i.e., for any
non-negative integer $m<n$ if for $\alpha$ and $\beta$ in
$\mathscr{I}^{\infty}_\lambda$ such that $\alpha\mathfrak{R}\beta$
and $$\min\left\{|\lambda\setminus\operatorname{dom}\alpha|,
|\lambda\setminus\operatorname{dom}\beta|\right\}=m$$ then
$\alpha=\beta$.

We consider two cases:
\begin{itemize}
  \item[$(i)$] $|\lambda\setminus\operatorname{dom}\alpha|\neq
   |\lambda\setminus\operatorname{dom}\beta|$; \; and
  \item[$(ii)$] $|\lambda\setminus\operatorname{dom}\alpha|=
   |\lambda\setminus\operatorname{dom}\beta|$.
\end{itemize}

Suppose case $(i)$ holds and
$|\lambda\setminus\operatorname{dom}\alpha|=n<
|\lambda\setminus\operatorname{dom}\beta|$. Then $\alpha$ and
$\beta$ are not $\mathscr{H}$-equivalent elements in
$\mathscr{I}^{\infty}_\lambda$ and hence by
Proposition~\ref{proposition-3.6} we obtain that
$\alpha\mathfrak{R}\gamma$ for all
$\gamma\in\mathscr{I}^{\infty}_\lambda$ with
$|\lambda\setminus\operatorname{dom}\gamma|\geqslant n$. Then
Proposition~\ref{proposition-3.8} implies that $\mu\mathfrak{R}\nu$
if and only if $\mu=\nu$ for all elements
$\mu,\nu\in\mathscr{I}^{\infty}_\lambda$ such that
$|\lambda\setminus\operatorname{dom}\mu|<n$ and
$|\lambda\setminus\operatorname{dom}\nu|<n$. Hence we get that
$\mathfrak{R}=\mathfrak{K}_n(I)$. We observe if $n=0$ then
$\mathfrak{R}=\Omega(\mathscr{I}^{\infty}_\lambda)$.

We henceforth assume that case $(ii)$ holds.

If $\alpha$ and $\beta$ are not $\mathscr{H}$-equivalent elements in
$\mathscr{I}^{\infty}_\lambda$ and then by
Proposition~\ref{proposition-3.6} we have that
$\alpha\mathfrak{R}\gamma$ for all
$\gamma\in\mathscr{I}^{\infty}_\lambda$ such that
$|\lambda\setminus\operatorname{dom}\gamma|\geqslant n$. Then
Proposition~\ref{proposition-3.8} implies that $\mu\mathfrak{R}\nu$
if and only if $\mu=\nu$ for all elements
$\mu,\nu\in\mathscr{I}^{\infty}_\lambda$ such that
$|\lambda\setminus\operatorname{dom}\mu|<n$ and
$|\lambda\setminus\operatorname{dom}\nu|<n$, and hence we have that
$\mathfrak{R}=\mathfrak{K}_n(I)$. Also in this case if $n=0$ then
$\mathfrak{R}=\Omega(\mathscr{I}^{\infty}_\lambda)$.

Suppose that $\alpha$ and $\beta$ are $\mathscr{H}$-equivalent
elements in $\mathscr{I}^{\infty}_\lambda$ and there exists no
non-$\mathscr{H}$-equivalent element $\delta$ of the semigroup
$\mathscr{I}^{\infty}_\lambda$ such that $\alpha\mathfrak{R}\delta$.
Otherwise by the previous part of the proof we have that
$\mathfrak{R}=\mathfrak{K}_n(I)$. Since
$(\alpha\cdot\alpha^{-1})\mathfrak{R}(\beta\cdot\alpha^{-1})$ we
conclude that without loss of generality we can assume that $\alpha$
is an identity element of $\mathscr{H}$-class $H(\alpha)$ which
contains $\alpha$ and $\beta\neq\alpha$. Since $\alpha$ is an
idempotent of the semigroup $\mathscr{I}^{\infty}_\lambda$ we have
that $\operatorname{dom}\alpha=\operatorname{ran}\alpha$ and the
restriction $\alpha|_{\operatorname{dom}\alpha}\colon
\operatorname{dom}\alpha\rightarrow\operatorname{dom}\alpha$ is an
identity map. Also we observe that the restriction of the partial
map $\beta|_{\operatorname{dom}\alpha}\colon
\operatorname{dom}\alpha\rightarrow\operatorname{dom}\alpha$ is a
permutation of the set $\operatorname{dom}\alpha$. Therefore without
loss of generality we can consider $\beta$ as a permutation of the
set $\operatorname{dom}\alpha$.

We consider two cases:
\begin{itemize}
  \item[$(1)$] $\beta$ is an odd permutation of the set
   $\operatorname{dom}\alpha$; \; \and
  \item[$(2)$] $\beta$ is an even permutation of the set
   $\operatorname{dom}\alpha$.
\end{itemize}

Suppose that $\beta$ is an odd permutation of the set
$\operatorname{dom}\alpha$. Since $H(\alpha)$ is a subgroup of the
semigroup $\mathscr{I}^{\infty}_\lambda$ we conclude that the
image $(H(\alpha))\Phi_\mathfrak{R}$ of $H(\alpha)$ is a subgroup
in $\mathscr{I}^{\infty}_\lambda/\mathfrak{R}$. Since the subgroup
$H(\alpha)$ is isomorphic to the group
$\textsf{S}_\infty(\lambda)$ and the group of all even
permutations $\textsf{A}_\infty(\lambda)$ of the set $\lambda$ is
a unique normal subgroup in $\textsf{S}_\infty(\lambda)$ (see
\cite[pp.~313--314, Example]{Guran1981} or
\cite{KarrasSolitar1956}) we conclude that the image
$(H(\alpha))\Phi_\mathfrak{R}$ is singleton. Then by
Theorem~2.20~\cite{CP} and
Proposition~\ref{proposition-2.2}~$(viii)$ for every
$\gamma\in\mathscr{I}^{\infty}_\lambda$ with
$|\lambda\setminus\operatorname{dom}\gamma|=
|\lambda\setminus\operatorname{dom}\alpha|$ the image
$(H_\gamma)\Phi_\mathfrak{R}$ of the $\mathscr{H}$-class
$H_\gamma$ which contains the element $\gamma$ is singleton and
hence by Propositions~\ref{proposition-3.6}, \ref{proposition-3.7}
and~\ref{proposition-3.8} we get that
$\mathfrak{R}=\mathfrak{K}_n(\textsf{S}_\infty)$.

Suppose that $\beta$ is an even permutation of the set
$\operatorname{dom}\alpha$. If the subgroup $H(\alpha)$ contains an
odd permutation $\delta$ of the set $\operatorname{dom}\alpha$ then
by previous proof we get that
$\mathfrak{R}=\mathfrak{K}_n(\textsf{S}_\infty)$. Suppose the
subgroup $H(\alpha)$ does not contain an odd permutation $\delta$ of
the set $\operatorname{dom}\alpha$. Since the subgroup $H(\alpha)$
is isomorphic to the group $\textsf{S}_\infty(\lambda)$ and the
group of all even permutations $\textsf{A}_\infty(\lambda)$ of the
set $\lambda$ is a unique normal subgroup in
$\textsf{S}_\infty(\lambda)$  we conclude that the image
$(H(\alpha))\Phi_\mathfrak{R}$ is a two-element subgroup in
$\mathscr{I}^{\infty}_\lambda/\mathfrak{R}$. Then by
Theorem~2.20~\cite{CP} and
Proposition~\ref{proposition-2.2}~$(viii)$ for every
$\gamma\in\mathscr{I}^{\infty}_\lambda$ with
$|\lambda\setminus\operatorname{dom}\gamma|=
|\lambda\setminus\operatorname{dom}\alpha|$ the image
$(H_\gamma)\Phi_\mathfrak{R}$ of the $\mathscr{H}$-class $H_\gamma$
which contains the element $\gamma$ is a two-element subset in
$\mathscr{I}^{\infty}_\lambda/\mathfrak{R}$ and hence by
Propositions~\ref{proposition-3.6}, \ref{proposition-3.7}
and~\ref{proposition-3.8} we get that
$\mathfrak{R}=\mathfrak{K}_n(\textsf{A}_\infty)$.
\end{proof}

\section{On topologizations of the free semilattice
$(\mathscr{P}_{<\omega}(\lambda),\cup)$}

\begin{definition}[{\cite{ChuchmanGutik2010?}}]\label{definitio-4.1}
We shall say that a semigroup $S$ has the
$\textsf{F}$-\emph{property} if for every $a,b,c,d\in S^1$ the sets
$\{x\in S\mid a\cdot x=b\}$ and $\{x\in S\mid x\cdot c=d\}$ are
finite or empty.
\end{definition}

Recall \cite{GHKLMS} an element $x$ of a semitopological semilattice
$S$ is a \emph{local minimum} if there exists an open neighbourhood
$U(x)$ of $x$ such that $U(x)\cap{\downarrow}x=\{x\}$. This is
equivalent to statement that ${\uparrow}x$ is an open subset in $S$.

A topological space $X$ is called \emph{Baire} if for each sequence
$A_1, A_2,\ldots, A_i,\ldots$ of nowhere dense subsets of $X$ the
union $\displaystyle\bigcup_{i=1}^\infty A_i$ is a co-dense subset
of $X$~\cite{Engelking1989}. A Tychonoff space $X$ is called
\emph{\v{C}ech complete} if for every compactification $cX$ of $X$
the remainder $cX\setminus c(X)$ is an $F_\sigma$-set in
$cX$~\cite{Engelking1989}.

A topological space $X$ is called \emph{hereditary Baire} if every
closed subset of $X$ is a Baire space~\cite{Engelking1989}. Every
\v{C}ech complete (and hence locally compact) space is hereditary
Baire (see \cite[Theorem~3.9.6]{Engelking1989}). We shall say that a
Hausdorff semitopological semigroup $S$ is an $I$-\emph{Baire space}
if for every $s\in S$ either $sS$ or $Ss$ is a Baire
space~\cite{ChuchmanGutik2010?}.

\begin{remark}\label{remark-4.2}
We observe that every left ideal $Ss$ and every right ideal $sS$ of
a regular semigroup $S$ is generated by its idempotents. Therefore
every principal left (right) ideal of a regular Hausdorff
semitopological semigroup $S$ is a closed subset of $S$. Hence every
regular Hausdorff hereditary Baire semitopological semigroup is a
$I$-Baire space.
\end{remark}

\begin{theorem}\label{theorem-4.3}
Let $S$ be a semilattice with the $\textsf{F}$-property. Then every
$I$-Baire topology $\tau$ on $S$ such that $(S,\tau)$ is a Hausdorff
semitopological semilattice is discrete.
\end{theorem}

\begin{proof}
Let $x$ be an arbitrary element of the semilattice $S$. We need to
show that $x$ is an isolated point in $(S,\tau)$.

Since $\tau$ is an $I$-Baire topology on $S$ we conclude that the
subspace ${\downarrow}x$ is Baire. We denote $S_x={\downarrow}x$.
For every positive integer $n$ we put
\begin{equation*}
    F_n=\{y\in S_x\mid |{\uparrow}y|=n\}.
\end{equation*}
Then we have that $S_x=\bigcup_{i=1}^\infty F_{n}$. Since the
topological space $S_x$ is Baire we conclude that that there exists
$F_{n}\in\mathscr{F}$ such that
$\operatorname{Int}_{S_x}(F_{n})\neq\varnothing$. We fix an
arbitrary $y_0\in\operatorname{Int}_{S_x}(F_{n})$. We observe that
the definition of the family $\{F_n\mid n\in\mathbb{N}\}$ implies
that for every non-empty subset $F_{n}$ and for any $s\in F_{n}$ the
sets ${\uparrow}s\cap F_{n}$ and ${\downarrow}s\cap F_{n}$ are
singleton. This implies that $y_0$ is a local minimum in $S_x$,
i.e., ${\uparrow}y_0$ is an open subset of $S$. Since the
semilattice $S_x$ has the $\textsf{F}$-property we conclude that the
Hausdorffness of $S$ implies that $x$ is an isolated point in $S_x$.
Then $x$ is a local minimum in $S$ and hence ${\uparrow}x$ is an
open subset in $S$. Since the semilattice $S$ has the
$\textsf{F}$-property we conclude that the Hausdorffness of $S$
implies that $x$ is an isolated point in $S$.
\end{proof}

\begin{remark}
We observe that the statement of Theorem~\ref{theorem-4.3} is true
for a $T_1$-semitopological $I$-Baire semilattice  with the
$\textsf{F}$-property.
\end{remark}

Since every \v{C}ech complete (and hence locally compact) space is
hereditary Baire, Theorem~\ref{theorem-4.3} implies the following
corollary:

\begin{corollary}\label{corollary-4.4}
Let $S$ be a semilattice with the $\textsf{F}$-property. Then every
\v{C}ech complete (locally compact) topology $\tau$ on $S$ such that
$(S,\tau)$ is a semitopological semilattice is discrete.
\end{corollary}

Since the free semilattice $(\mathscr{P}_{<\omega}(\lambda),\cup)$
has $\textsf{F}$-property, Theorem~\ref{theorem-4.3} implies the
following corollary:

\begin{corollary}\label{corollary-4.5}
Every Hausdorff $I$-Baire (\v{C}ech complete, locally compact)
topology $\tau$ on the free semilattice
$\mathscr{P}_{<\omega}(\lambda)$ such that
$(\mathscr{P}_{<\omega}(\lambda),\tau)$ is a semitopological
semilattice is discrete.
\end{corollary}


\section{On a topological semigroup
$\mathscr{I}^{\infty}_\lambda$}

\begin{theorem}\label{theorem-5.1}
Every hereditary Baire topology $\tau$ on the semigroup
$\mathscr{I}^{\infty}_\omega$ such that
$(\mathscr{I}^{\infty}_\omega,\tau)$ is a Hausdorff semitopological
semigroup is discrete.
\end{theorem}

\begin{proof}
Let $\alpha$ be an arbitrary element of the the semigroup
$\mathscr{I}^{\infty}_\omega$. We need to show that $\alpha$ is an
isolated point in $(\mathscr{I}^{\infty}_\omega,\tau)$.

For every non-negative integer $n$ we denote
$C_n=\mathscr{I}^{\infty}_\omega\setminus I_{n+1}$.

By induction we shall prove that for every non-negative integer $n$
the following statement holds: \emph{every $\alpha\in C_n$ is an
isolated point in $(\mathscr{I}^{\infty}_\omega,\tau)$}.

First we shall show that our statement is true for $n=0$. We define
a family
$\mathscr{C}=\{\{{\beta\}}\mid\beta\in\mathscr{I}^{\infty}_\omega\}$.
Since the topological space $(\mathscr{I}^{\infty}_\omega,\tau)$ is
Baire we have that the family $\mathscr{C}$ has an element with
non-empty interior and hence the topological space
$(\mathscr{I}^{\infty}_\omega,\tau)$ has an isolated point $\gamma$
in $(\mathscr{I}^{\infty}_\omega,\tau)$. Then
$|\omega\setminus\operatorname{dom}\alpha|=0$ and hence statements
$(viii)-(xi)$ of Proposition~\ref{proposition-2.2} imply that there
exist $\mu,\nu\in\mathscr{I}^{\infty}_\omega$ such that
$\mu\cdot\alpha\cdot\nu=\gamma$. Since translations in
$(\mathscr{I}^{\infty}_\omega,\tau)$ are continuous we conclude that
Hausdorffness of the space $(\mathscr{I}^{\infty}_\omega,\tau)$ and
Proposition~\ref{proposition-2.5} imply that $\alpha$ an isolated
point in $(\mathscr{I}^{\infty}_\omega,\tau)$.

Suppose our statement is true for all $n<k$, $k\in\mathbb{N}$. We
shall show that its is true for $n=k$. Our assumption implies that
$I_{k}$ is a closed subset of $(\mathscr{I}^{\infty}_\omega,\tau)$.
Later we shall denote by $\tau_k$ the topology induced from
$(\mathscr{I}^{\infty}_\omega,\tau)$ onto $I_{k}$. Then
$(I_{k},\tau_k)$ is a Baire space. We define a family
$\mathscr{C}_k=\{\{{\beta\}}\mid\beta\in I_k\}$. Since the
topological space $(I_{k},\tau_k)$ is Baire we have that the family
$\mathscr{C}_k$ has an element with non-empty interior and hence the
topological space $(I_{k},\tau_k)$ has an isolated point $\gamma$ in
$(I_{k},\tau_k)$. Let $U(\gamma)$ be an open neighbourhood
$U(\gamma)$ of $\gamma$ in $(\mathscr{I}^{\infty}_\omega,\tau)$ such
that $U(\gamma)\cap I_k=\{\gamma\}$. Since
$(\mathscr{I}^{\infty}_\omega,\tau)$ is a semitopological semigroup
we have that there exists an open neighbourhood $V(\gamma)$ of
$\gamma$ in $(\mathscr{I}^{\infty}_\omega,\tau)$ such that
$V(\gamma)\subseteq U(\gamma)$ and $\gamma\cdot\gamma^{-1}\cdot
V(\gamma)\subseteq U(\gamma)$. We remark that
$\gamma\cdot\gamma^{-1}\cdot V(\gamma)\subseteq\{\gamma\}$. Hence by
Proposition~\ref{proposition-2.5} the neighbourhood $V(\gamma)$ is
finite and Hausdorffness of the space
$(\mathscr{I}^{\infty}_\omega,\tau)$ implies that $\gamma$ an
isolated point in $(\mathscr{I}^{\infty}_\omega,\tau)$. Let $\alpha$
be an arbitrary element of the set $I_k\setminus I_{k+1}$. Then
$|\omega\setminus\operatorname{dom}\alpha|=k$ and hence statements
$(viii)-(xi)$ of Proposition~\ref{proposition-2.2} imply that there
exist $\mu,\nu\in\mathscr{I}^{\infty}_\omega$ such that
$\mu\cdot\alpha\cdot\nu=\gamma$. Since translations in
$(\mathscr{I}^{\infty}_\omega,\tau)$ are continuous we conclude that
Hausdorffness of the space $(\mathscr{I}^{\infty}_\omega,\tau)$ and
Proposition~\ref{proposition-2.5} imply that $\alpha$ an isolated
point in $(\mathscr{I}^{\infty}_\omega,\tau)$. This completes the
proof of our theorem.
\end{proof}

\begin{remark}\label{remark-5.2}
We observe that the statement of Theorem~\ref{theorem-5.1} holds for
every topology $\tau$ on the semigroup $\mathscr{I}^{\infty}_\omega$
such that $(\mathscr{I}^{\infty}_\omega,\tau)$ is a Hausdorff
semitopological semigroup and every (two-sided) ideal in
$(\mathscr{I}^{\infty}_\omega,\tau)$ is a Baire space.
\end{remark}

Theorem~\ref{theorem-5.1} implies the following corollary:

\begin{corollary}\label{corollary-5.3}
Every \v{C}ech complete (locally compact) topology $\tau$ on the
semigroup $\mathscr{I}^{\infty}_\omega$ such that
$(\mathscr{I}^{\infty}_\omega,\tau)$ is a Hausdorff semitopological
semigroup is discrete.
\end{corollary}

\begin{theorem}\label{theorem-5.4}
Let $\lambda$ be an infinite cardinal and $S$ be a topological
semigroup which contains a dense discrete subsemigroup
$\mathscr{I}^{\infty}_\lambda$. If
$I=S\setminus\mathscr{I}^{\infty}_\lambda \neq\varnothing$ then $I$
is an ideal of $S$.
\end{theorem}

\begin{proof}
Suppose that $I$ is not an ideal of $S$. Then at least one of the
following conditions holds:
\begin{equation*}
    1)~I\cdot\mathscr{I}^{\infty}_\lambda\nsubseteq I,
    \qquad 2)~\mathscr{I}^{\infty}_\lambda\cdot I\nsubseteq I, \qquad \mbox{or}
    \qquad 3)~I\cdot I\nsubseteq I.
\end{equation*}
Since $\mathscr{I}^{\infty}_\lambda$ is a discrete dense subspace of
$S$, Theorem~3.5.8~\cite{Engelking1989} implies that
$\mathscr{I}^{\infty}_\lambda$ is an open subspace of $S$. Suppose
there exist $a\in\mathscr{I}^{\infty}_\lambda$ and $b\in I$ such
that $b\cdot a=c\notin I$. Since $\mathscr{I}^{\infty}_\lambda$ is a
dense open discrete subspace of $S$ the continuity of the semigroup
operation in $S$ implies that there exists an open neighbourhood
$U(b)$ of $b$ in $S$ such that $U(b)\cdot \{a\}=\{c\}$. But by
Proposition~\ref{proposition-2.5} the equation  $x\cdot a=c$ has
finitely many solutions in $\mathscr{I}^{\infty}_\lambda$. This
contradicts the assumption that $b\in
S\setminus\mathscr{I}^{\infty}_\lambda$. Therefore $b\cdot a=c\in I$
and hence $I\cdot\mathscr{I}^{\infty}_\lambda\subseteq I$. The proof
of the inclusion $\mathscr{I}^{\infty}_\lambda\cdot I\subseteq I$ is
similar.

Suppose there exist $a,b\in I$ such that $a\cdot b=c\notin I$. Since
$\mathscr{I}^{\infty}_\lambda$ is a dense open discrete subspace of
$S$ the continuity of the semigroup operation in $S$ implies that
there exist open neighbourhoods $U(a)$ and $U(b)$ of $a$ and $b$ in
$S$, respectively, such that $U(a)\cdot U(b)=\{c\}$. But by
Proposition~\ref{proposition-2.5} the equations $x\cdot b_0=c$ and
$a_0\cdot y=c$ have finitely many solutions in
$\mathscr{I}^{\infty}_\lambda$. This contradicts the assumption that
$a,b\in S\setminus\mathscr{I}^{\infty}_\lambda$. Therefore $a\cdot
b=c\in I$ and hence $I\cdot I\subseteq I$.
\end{proof}

\begin{proposition}\label{proposition-5.5}
Let $S$ be a topological semigroup which contains a dense discrete
subsemigroup $\mathscr{I}^{\infty}_\lambda$. Then for every
$c\in\mathscr{I}^{\infty}_\lambda$ the set
\begin{equation*}
    D_c(\mathscr{I}^{\infty}_\lambda)=
    \{(x,y)\in\mathscr{I}^{\infty}_\lambda
    \times\mathscr{I}^{\infty}_\lambda\mid x\cdot y=c\}
\end{equation*}
is a closed-and-open subset of $S\times S$.
\end{proposition}

\begin{proof}
Since $\mathscr{I}^{\infty}_\lambda$ is a discrete subspace of $S$
we have that $D_c(\mathscr{I}^{\infty}_\lambda)$ is an open subset
of $S\times S$.

Suppose that there exists $c\in\mathscr{I}^{\infty}_\lambda$ such
that $D_c(\mathscr{I}^{\infty}_\lambda)$ is a non-closed subset of
$S\times S$. Then there exists an accumulation point $(a,b)\in
S\times S$ of the set $D_c(\mathscr{I}^{\infty}_\lambda)$. The
continuity of the semigroup operation in $S$ implies that $a\cdot
b=c$. But $\mathscr{I}^{\infty}_\lambda\times
\mathscr{I}^{\infty}_\lambda$ is a discrete subspace of $S\times S$
and hence by Theorem~\ref{theorem-5.4} the points $a$ and $b$ belong
to the ideal $I=S\setminus \mathscr{I}^{\infty}_\lambda$ and hence
$a\cdot b\in S\setminus \mathscr{I}^{\infty}_\lambda$ cannot be
equal to $c$.
\end{proof}

A topological space $X$ is defined to be \emph{pseudocompact} if
each locally finite open cover of $X$ is finite. According to
\cite[Theorem~3.10.22]{Engelking1989} a Tychonoff topological space
$X$ is pseudocompact if and only if each continuous real-valued
function on $X$ is bounded.

\begin{theorem}\label{theorem-5.6}
If a topological semigroup $S$ contains
$\mathscr{I}^{\infty}_\lambda$ as a dense discrete subsemigroup then
the square $S\times S$ is not pseudocompact.
\end{theorem}

\begin{proof}
Since the square $S\times S$ contains an infinite closed-and-open
discrete subspace $D_c(\mathscr{I}^{\infty}_\lambda)$, we conclude
that $S\times S$ fails to be pseudocompact (see
\cite[Ex.~3.10.F(d)]{Engelking1989} or \cite{Colmez1951}).
\end{proof}

A topological space $X$ is called \emph{countably compact} if any
countable open cover of $X$ contains a finite
subcover~\cite{Engelking1989}. We observe that every Hausdorff
countably compact space is pseudocompact.

Since the closure of an arbitrary subspace of a countably compact
space is countably compact (see
\cite[Theorem~3.10.4]{Engelking1989}) Theorem~\ref{theorem-5.6}
implies the following corollary:

\begin{corollary}\label{corollary-5.7}
For every infinite cardinal $\lambda$ the discrete semigroup
$\mathscr{I}^{\infty}_\lambda$ does not embed into a topological
semigroup $S$ with the countably compact square $S\times S$.
\end{corollary}

Since every compact topological space is countably compact
Theorem~3.24~\cite{Engelking1989} and Corollary~\ref{corollary-5.7}
imply

\begin{corollary}\label{corollary-5.8}
For every infinite cardinal $\lambda$ the discrete semigroup
$\mathscr{I}^{\infty}_\lambda$ does not embed into a compact
topological semigroup.
\end{corollary}

We recall that the Stone-\v{C}ech compactification of a Tychonoff
space $X$ is a compact Hausdorff space $\beta{X}$ containing $X$ as
a dense subspace so that each continuous map $f\colon X\rightarrow
Y$ to a compact Hausdorff space $Y$ extends to a continuous map
$\overline{f}\colon \beta{X}\rightarrow Y$~\cite{Engelking1989}.

\begin{theorem}\label{theorem-5.9}
For every infinite cardinal $\lambda$ the discrete semigroup
$\mathscr{I}^{\infty}_\lambda$ does not embed into a Tychonoff
topological semigroup $S$ with the pseudocompact square $S\times S$.
\end{theorem}

\begin{proof}
By Theorem~1.3~\cite{BanakhDimitrova2010} for any topological
semigroup $S$ with the pseudocompact square $S\times S$ the
semigroup operation $\mu\colon S\times S\rightarrow S$ extends to a
continuous semigroup operation $\beta\mu\colon\beta S\times\beta
S\rightarrow\beta S$, so $S$ is a subsemigroup of the compact
topological semigroup $\beta S$. Then Corollary~\ref{corollary-5.8}
implies the statement of the theorem.
\end{proof}


The following example shows that there exists a non-discrete
topology $\tau_F$ on the semigroup $\mathscr{I}^{\infty}_\lambda$
such that $(\mathscr{I}^{\infty}_\lambda,\tau_F)$ is a Tychonoff
topological inverse semigroup.

\begin{example}\label{example-5.10}
We define a topology $\tau_F$ on the semigroup
$\mathscr{I}^{\infty}_\lambda$ as follows. For every
$\alpha\in\mathscr{I}^{\infty}_\lambda$ we define a family
\begin{equation*}
    \mathscr{B}_F(\alpha)=\{U_\alpha(F)\mid F \mbox{ is a finite
    subset of } \operatorname{dom}\alpha\},
\end{equation*}
where
\begin{equation*}
    U_\alpha(F)=\{\beta\in\mathscr{I}^{\infty}_\lambda
    \mid \operatorname{dom}\alpha=\operatorname{dom}\beta,
    \operatorname{ran}\alpha=\operatorname{ran}\beta \mbox{ and }
    (x)\beta=(x)\alpha \mbox{ for all } x\in F\}.
\end{equation*}
Since conditions (BP1)--(BP3)~\cite{Engelking1989} hold for the
family $\{\mathscr{B}_F(\alpha)\}_{\alpha\in
\mathscr{I}^{\infty}_\lambda}$ we conclude that the family
$\{\mathscr{B}_F(\alpha)\}_{\alpha\in
\mathscr{I}^{\infty}_\lambda}$ is the base of the topology
$\tau_F$ on the semigroup $\mathscr{I}^{\infty}_\lambda$.
\end{example}

\begin{proposition}\label{proposition-5.11}
$(\mathscr{I}^{\infty}_\lambda,\tau_F)$ is a Tychonoff topological
inverse semigroup.
\end{proposition}

\begin{proof}
Let $\alpha$ and $\beta$ be arbitrary elements of the semigroup
$\mathscr{I}^{\infty}_\lambda$. We put $\gamma=\alpha\cdot\beta$ and
let $F=\{n_1,\ldots,n_i\}$ be a finite subset of
$\operatorname{dom}\gamma$. We denote
$m_1=(n_1)\alpha,\ldots,m_i=(n_i)\alpha$ and
$k_1=(n_1)\gamma,\ldots,$ $k_i=(n_i)\gamma$. Then we get that
$(m_1)\beta=k_1,\ldots,(m_i)\beta=k_i$. Hence we have that
\begin{equation*}
    U_\alpha(\{n_1,\ldots,n_i\})\cdot
    U_\beta(\{m_1,\ldots,m_i\})\subseteq
    U_\gamma(\{n_1,\ldots,n_i\})
\end{equation*}
and
\begin{equation*}
    \big(U_\gamma(\{n_1,\ldots,n_i\})\big)^{-1}\subseteq
    U_{\gamma^{-1}}(\{k_1,\ldots,k_i\}).
\end{equation*}
Therefore the semigroup operation and the inversion are continuous
in
$(\mathscr{I}_{\infty}^{\,\Rsh\!\!\!\nearrow}(\mathbb{N}),\tau_F)$.

We observe that the group of units $H(\mathbb{I})$ of the
semigroup $\mathscr{I}^{\infty}_\lambda$ with the induced topology
$\tau_F(H(\mathbb{I}))$ from
$(\mathscr{I}^{\infty}_\lambda,\tau_F)$ is a topological group
(see \cite[pp.~313--314, Example]{Guran1981} or
\cite{KarrasSolitar1956}) and the definition of the topology
$\tau_F$ implies that every $\mathscr{H}$-class of the semigroup
$\mathscr{I}^{\infty}_\lambda$ is an open-and-closed subset of the
topological space $(\mathscr{I}^{\infty}_\lambda,\tau_F)$.
Therefore Theorem~2.20~\cite{CP} implies that the topological
space $(\mathscr{I}^{\infty}_\lambda,\tau_F)$ is homeomorphic to a
countable topological sum of topological copies of
$\big(H(\mathbb{I}),\tau_F(H(\mathbb{I}))\big)$. Since every
$T_0$-topological group is a Tychonoff topological space (see
\cite[Theorem~3.10]{Pontryagin1966} or
\cite[Theorem~8.4]{HewittRoos1963}) we conclude that the
topological space $(\mathscr{I}^{\infty}_\lambda,\tau_F)$ is
Tychonoff too. This completes the proof of the proposition.
\end{proof}

\begin{remark}\label{remarke-5.12}
We observe that the topology $\tau_{F}$ on
$\mathscr{I}^{\infty}_\lambda$ induces the discrete topology on the
band $E(\mathscr{I}^{\infty}_\lambda)$.
\end{remark}

\begin{example}\label{example-5.13}
We define a topology $\tau_{W\!F}$ on the semigroup
$\mathscr{I}^{\infty}_\lambda$ as follows. For every
$\alpha\in\mathscr{I}^{\infty}_\lambda$ we define a family
\begin{equation*}
    \mathscr{B}_{W\!F}(\alpha)=\{U_\alpha(F)\mid F \mbox{ is a finite
    subset of } \operatorname{dom}\alpha\},
\end{equation*}
where
\begin{equation*}
    U_\alpha(F)=\{\beta\in\mathscr{I}^{\infty}_\lambda
    \mid \operatorname{dom}\beta\subseteq\operatorname{dom}\alpha \mbox{ and }
    (x)\beta=(x)\alpha \mbox{ for all } x\in F\}.
\end{equation*}
Since conditions (BP1)--(BP3)~\cite{Engelking1989} hold for the
family $\{\mathscr{B}_{W\!F}(\alpha)\}_{\alpha\in
\mathscr{I}^{\infty}_\lambda}$ we conclude that the family
$\{\mathscr{B}_{W\!F}(\alpha)\}_{\alpha\in
\mathscr{I}^{\infty}_\lambda}$ is the base of the topology
$\tau_{W\!F}$ on the semigroup $\mathscr{I}^{\infty}_\lambda$.
\end{example}

\begin{proposition}\label{proposition-5.14}
$(\mathscr{I}^{\infty}_\lambda,\tau_{W\!F})$ is a Hausdorff
topological inverse semigroup.
\end{proposition}

\begin{proof}
Let $\alpha$ and $\beta$ be arbitrary elements of the semigroup
$\mathscr{I}^{\infty}_\lambda$. We put $\gamma=\alpha\cdot\beta$ and
let $F=\{n_1,\ldots,n_i\}$ be a finite subset of
$\operatorname{dom}\gamma$. We denote
$m_1=(n_1)\alpha,\ldots,m_i=(n_i)\alpha$ and
$k_1=(n_1)\gamma,\ldots$, $k_i=(n_i)\gamma$. Then we get that
$(m_1)\beta=k_1,\ldots,(m_i)\beta=k_i$. Hence we have that
\begin{equation*}
    U_\alpha(\{n_1,\ldots,n_i\})\cdot
    U_\beta(\{m_1,\ldots,m_i\})\subseteq
    U_\gamma(\{n_1,\ldots,n_i\})
\end{equation*}
and
\begin{equation*}
    \big(U_\gamma(\{n_1,\ldots,n_i\})\big)^{-1}\subseteq
    U_{\gamma^{-1}}(\{k_1,\ldots,k_i\}).
\end{equation*}
Therefore the semigroup operation and the inversion are continuous
in $(\mathscr{I}^{\infty}_\lambda,\tau_{W\!F})$.

Later we shall show that the topology $\tau_{W\!F}$ is Hausdorff.
Let $\alpha$ and $\beta$ be arbitrary distinct points of the space
$(\mathscr{I}^{\infty}_\lambda,\tau_{W\!F})$. Then only one of the
following conditions holds:
\begin{itemize}
    \item[$(i)$]
    $\operatorname{dom}\alpha=\operatorname{dom}\beta$;

    \item[$(ii)$]
    $\operatorname{dom}\alpha\neq\operatorname{dom}\beta$.
\end{itemize}

In case $\operatorname{dom}\alpha=\operatorname{dom}\beta$ we have
that there exists $x\in\operatorname{dom}\alpha$ such that
$(x)\alpha\neq(x)\beta$. The definition of the topology
$\tau_{W\!F}$ implies that $U_\alpha(\{x\})\cap
U_\beta(\{x\})=\varnothing$.

If $\operatorname{dom}\alpha\neq\operatorname{dom}\beta$, then
only one of the following conditions holds:
\begin{itemize}
    \item[$(a)$]
    $\operatorname{dom}\alpha\subsetneqq\operatorname{dom}\beta$;

    \item[$(b)$]
    $\operatorname{dom}\beta\subsetneqq\operatorname{dom}\alpha$;

    \item[$(c)$]
    $\operatorname{dom}\alpha\setminus\operatorname{dom}\beta\neq\varnothing$
    and
    $\operatorname{dom}\beta\setminus\operatorname{dom}\alpha\neq\varnothing$.
\end{itemize}

Suppose that case $(a)$ holds. Let be
$x\in\operatorname{dom}\beta\setminus\operatorname{dom}\alpha$ and
$y\in\operatorname{dom}\alpha$. The definition of the topology
$\tau_{W\!F}$ implies that  $U_\alpha(\{y\})\cap
U_\beta(\{x\})=\varnothing$.

Case $(b)$ is similar to $(a)$.

Suppose that case $(c)$ holds. Let be
$x\in\operatorname{dom}\beta\setminus\operatorname{dom}\alpha$ and
$y\in\operatorname{dom}\alpha\setminus\operatorname{dom}\beta$.
The definition of the topology $\tau_{W\!F}$ implies that
$U_\alpha(\{y\})\cap U_\beta(\{x\})=\varnothing$.

This completes the proof of the proposition.
\end{proof}

\begin{remark}
We observe that the topology $\tau_{W\!F}$ on
$\mathscr{I}^{\infty}_\lambda$ induces a non-discrete topology (and
hence a non-hereditary Baire topology) on the band
$E(\mathscr{I}^{\infty}_\lambda)$. Moreover, $\mathscr{H}$-classes
in $(\mathscr{I}^{\infty}_\lambda,\tau_{W\!F})$ and
$(\mathscr{I}^{\infty}_\lambda,\tau_{F})$ are homeomorphic
subspaces.
\end{remark}


\section*{Acknowledgements}

We acknowledge Taras Banakh for his comments and suggestions. The
authors are grateful to the referee for several comments and
suggestions which have considerably improved the original version of
the manuscript.


\end{document}